\documentclass[12pt]{article}%
\usepackage{amsmath,amssymb,amsfonts,amsthm}
\usepackage{setspace}
\usepackage{fullpage}
\usepackage{hyperref}
\usepackage{amscd}
\usepackage{graphpap}
\usepackage[pdftex]{color,graphicx}
\usepackage{color,graphicx}
\usepackage{empheq}
\usepackage{ragged2e}
\usepackage{comment}
\usepackage{array,epsfig}
\usepackage{amsxtra}
\usepackage{mathrsfs}
\usepackage{color}
\usepackage{makeidx}
\usepackage[justification=centering]{caption}
\usepackage{enumerate}
\usepackage{nccmath}
\usepackage{empheq}
\usepackage{leading} 
\usepackage{subfigure} 
\usepackage{graphicx}
\usepackage{caption,multicol,booktabs,array}
\usepackage[singlelinecheck=false]{caption}
\usepackage[labelfont=bf,labelsep=quad,justification=justified]{caption}
\usepackage{array}
\usepackage{float}
\usepackage{caption}
\usepackage{booktabs}
\usepackage{arydshln}
\usepackage{hyperref}

\newtheorem{theorem}{Theorem}[section]
\newtheorem{definition}{Definition}[section]

\newtheorem{proposition}[theorem]{Proposition}
\newtheorem{corollary}[theorem]{Corollary}

\newtheorem{example}{Example}

\usepackage{geometry}
 \hypersetup{
  colorlinks=true,
  linkcolor=blue,
  citecolor=red,
  filecolor=blue,
  urlcolor=blue,
}

\title{\bf  Convergence and inference for mixed Poisson random sums}
\author{W. Barreto-Souza$^\star$\footnote{Email: \texttt{wagner.barretosouza@kaust.edu.sa}},\,\, G. Oliveira$^\star$\footnote{Email: \texttt{gabriela.oliveira.mat@gmail.com}}\,\, and R.W.C. Silva$^\star$\footnote{Email: \texttt{rogerwcs@est.ufmg.br}} \\\\
\small $^\star$\it Departamento de Estat\' \i stica, Universidade Federal de Minas Gerais, Belo Horizonte, Brazil\\
\small $^\ast$\it Statistics Program, King Abdullah University of Science and Technology, Thuwal, Saudi Arabia}
\date{}

\begin{document}

\maketitle

\begin{abstract}
In this paper we obtain the limit distribution for partial sums with a random number of terms following a class of mixed Poisson distributions. The resulting weak limit is a mixing between a normal distribution and an exponential family, which we call by normal exponential family (NEF) laws. A new stability concept is introduced and a relationship between $\alpha$-stable distributions and NEF laws is established. We propose estimation of the parameters of the NEF models through the method of moments and also by the maximum likelihood method, which is performed via an Expectation-Maximization algorithm. Monte Carlo simulation studies are addressed to check the performance of the proposed estimators and an empirical illustration on financial market is presented.
     \\
     
\noindent {\bf Keywords:}  EM-algorithm; Mixed Poisson distribution; Stability; Weak convergence.

\end{abstract}

\section{Introduction}
One of the most important and beautiful theorems in probability theory is the Central Limit Theorem, which lays down the convergence in distribution of the partial sum (properly normalized) of {\it i.i.d.} random variables with finite second moment to a normal distribution. This can be seen as a characterization of the normal distribution as the weak limit of such sums. A natural variant of this problem is placed when the number of terms in the sum is random. For instance, counting processes are of fundamental importance in the theory of probability and statistics. A comprehensive account for this topic is given in  \cite{gnekor1996}. One of the earliest counting models is the compound Poisson process $\{C_t\}_{t\geq 0}$ defined as 
\begin{equation}\label{cpoisson}C_t=\sum_{n=1}^{N_t}X_n,\quad t\geq 0,
\end{equation} 
where $\{N_t\}_{t\geq 0}$ is a Poisson process with rate $\lambda t$, $\lambda>0$, and $\{X_n\}_{n\in\mathbb{N}}$ is a sequence of ${\it i.i.d.}$ random variables independent of the Poisson process. Applications of the random summation (\ref{cpoisson}) include risk theory, biology, queuing theory and finance; for instance, see \cite{ekm}, \cite{p} and \cite{pb}. For fixed $t$, it can be shown that the random summation given in (\ref{cpoisson}), when properly normalized, converges weakly to the standard normal distribution as $\lambda \rightarrow \infty$. 

Another important quantity is the geometric random summation defined as
\begin{equation*}\label{geosum}S_p=\sum_{n=1}^{\nu_p}X_n,
\end{equation*} 
where $\nu_p$ is a geometric random variable with probability function $P(\nu_p=k)=(1-p)^{k-1}p,\,\,\,k=1,2,\dots,$ and  $\{X_n\}_{n\in\mathbb{\,\,N}}$ is a sequence of {\it i.i.d.} random variables independent of $\nu_p$, for $p\in(0,1)$. Geometric summation has a wide range of applications such as risk theory, modeling financial asset returns, insurance mathematics and others, as discussed in \cite{k}.

In \cite{r} it is shown that if $X_n$ is a positive random variable with finite mean, then $pS_p$ converges weakly to an exponential distribution as $p\rightarrow 0$. If the $X_n's$ are symmetric with $E(X_1)=0$ and finite second moment, then there exists $a_p$ such that $a_p S_p$ converges weakly to a Laplace distribution when $p\rightarrow 0$. If $X_n$ has an asymmetric distribution, it is possible to show that the geometric summation properly normalized converges in distribution to the asymmetric Laplace distribution. These last two results and their proofs can be found in \cite{kkp}. 

The purpose of the present paper is to study the random summation with mixed Poisson number of terms. For a review about mixed Poisson distributions see \cite{kx}. In \cite{gavkor2006} it is shown that the mixed Poisson (MP) random sum converges weakly to a scale mixture of normal distributions (see \cite{west1987} for a definition of such mixture) by assuming that the sequence $\{X_n\}_{n\geq1}$ is {\it i.i.d.} (with $E(X_1)=0$ and $\mbox{Var}(X_1)=1$) and that there exists $\delta>0$ such that $E(|X_1|^{2+\delta})<\infty$. This last assumption is necessary since the main interest in that paper is to find a Berry-Eessen type bound for the weak convergence. The study of accuracy for the convergence of MP random sums is also considered in \cite{korshe2012}, \cite{kordor2017} and \cite{she2018}. Limit theorems for random summations with a negative binomial or generalized negative binomial (which are MP distributions) number of terms, with applications to real practical situations, are addressed in \cite{benkor2005}, \cite{sc} and \cite{korzei2019}.

Our chief goal in this paper is to explore mixed Poisson random summations under different assumptions compared to those in previous works in the literature, since our aims here are also different. We  assume that the number of terms follows the MP class of distributions proposed in \cite{barsou2015} and \cite{bs}. This class contains the negative binomial and Poisson inverse-Gaussian distributions as particular cases. Further, we assume that the sequence $\{X_n\}_{n\geq1}$ is {\it i.i.d.} with non-null mean and finite second moment. We do not require more than finite second moment, in contrast to the work in \cite{gavkor2006}. Under these conditions, we show that the weak limit of a MP random sum belongs to a class of normal variance-mean mixtures (see \cite{bn} for a definition of this kind of distribution) driven by a latent exponential family. We call this new class of distributions by normal-exponential family (in short NEF). In particular, this class contains the normal inverse-Gaussian (NIG) distribution introduced in \cite{bn} as a special case. Therefore, this provides a new characterization for the NIG law.

Another contribution of this paper is the introduction of the new mixed Poisson stability concept, which includes the geometric stability (see \cite{kr,mr}) as a particular case. We also provide a theorem establishing a relationship between our proposed MP stability and the $\alpha$-stable distributions.

The statistical contribution of our paper is the inferential study on the limiting class of distributions, which is of practical interest. We propose estimation of the parameters of the NEF models through the method of moments and also by the maximum likelihood method, which is performed via an Expectation-Maximization (EM) algorithm (see \cite{demetal1977}).

The paper is organized in the following manner. In Section \ref{theo_results} we show that the mixed Poisson random sums converges weakly, under some mild conditions, to a normal variance-mean mixture. Further, we define a new concept called mixed Poisson stability, which generalizes the well-known geometric stability. Properties of the limiting class of NEF distributions are explored in Section \ref{lim_distributions}. Inferential aspects of the NEF models are addressed in Section \ref{inference}. In Section \ref{simulation} we present Monte Carlo simulations to check the finite-sample behavior of the proposed estimators. A real data application is presented in Section \ref{application}.

\section{Weak convergence and stability}\label{theo_results}

In this section we provide the main probabilistic results of the paper. To do this, we first present some preliminaries about the mixed Poisson distributions considered here. Then, we establish the weak convergence for mixed Poisson summations and based on this we introduce a new stability concept.

\subsection{Weak limit of mixed Poisson random sums}

A mixed Poisson distribution is a generalization of the Poisson distribution which is constructed as follows. 

\begin{definition} Let $W_\phi$ be a strictly positive random variable with distribution function $G_{\phi}(\cdot)$, where $\phi$ denotes a parameter associated to $G$. We will later assume $W_\phi$ belongs to a particular exponential family of distributions. Let $N|W_{\phi}=w \sim  $ Poisson $(\lambda w)$, for  $\lambda>0$. In this case we say that $N$ follows a mixed Poisson distribution. Its probability function assumes the form
\begin{equation*}\label{probabilidade poisson misturada}
P(N = n) = \int_0^{\infty} \dfrac{e^{-\lambda w}(\lambda w)^n}{n!}dG_{\phi}(w),\quad \mbox{for}\quad n\in\mathbb N\equiv\{0,1,2,\ldots\}.
\end{equation*}
\end{definition}

For instance, if $W_\phi$ is assumed to be gamma or inverse-Gaussian distributed, then $N$ is negative binomial or Poisson inverse-Gaussian distributed, respectively. 

We consider the class of mixed Poisson distributions introduced in \cite{bs}, which is defined by assuming that $W_\phi$ is a continuous positive random variable belonging to the exponential family of distributions. This family was also considered in a survival analysis context in \cite{barsou2015}. We assume that there exist a $\sigma$-finite measure $\nu$ such that the probability density function (pdf) of $W_\phi$ with respect to $\nu$ is 
\begin{equation}\label{EF_density}
f_{W_\phi}(w) = \exp\{\phi[w\xi_0 - b(\xi_0)] + c(w; \phi)\}, \quad w > 0, \quad \phi > 0,
\end{equation}
where $b(\cdot)$ is continuous, three times differentiable and $\xi_0$ is such that  $b'  (\xi_0) = 1$ and $c(\cdot,\cdot): \mathbb{R}^{+} \times \mathbb{R}^{+} \to \mathbb{R}$. In this case, $E(W) = b'(\xi_0) = 1$ and $Var(W) = \phi^{-1}b''(\xi_0)$. For more details about this class of distributions we refer the reader to \cite{bs}.

From now on we adopt the following notation: for any random variable $X$ we write $\psi_X(t)$ for its characteristic function ($ch.f.$). We write $W_{\phi} \sim\mbox{EF}(\phi)$ for $W_{\phi}$ belonging to the exponential family and $N_{\lambda}\sim\mbox{MP}(\lambda, W_{\phi})$, making clear the mixture distribution involved; we also denote $N_{\lambda}\sim MP(\lambda, \phi)$ when the latent variable is not important for the discussion in question. Let $S_{\lambda}\equiv X_1+X_2+\cdots+X_{N_{\lambda}}$, where $N_{\lambda}\sim MP(\lambda,W_{\phi})$ as before and $S_{\lambda}\equiv0$ when $N_{\lambda}=0$. Throughout the text $\{X_n\}_{n\in\mathbb{N}}$ will always be a sequence of $i.i.d.$ random variables independent of $N_{\lambda}$. 

Before we can state our main result we need an extra observation. In \cite{s} the author provides a characterization of the exponential familty with a single natural parameter $\theta$ in terms of its characteristic function. In that paper, $T_{\theta}$ belongs to this family if there exists a $\sigma$-finite measure $\nu$ such that the pdf of $T_{\theta}$ with respect to $\nu$ is of the form
\begin{equation}\label{EF_density2}
f_{T_{\theta}}(y) = \exp[\theta y+Q(\theta)+R(y)],\quad y\in\mathbb S,
\end{equation}
where $\mathbb S$ is the support of the distribution. The following theorem appears in \cite{s} and plays an important role in this paper.

\begin{theorem}[Sampson, A.R.]\label{sampson}
	Let $\{T_{\theta},\,\theta\in (a,b)\}$ be a family of random variables such that (\ref{EF_density2}) holds and  $E(T_{\theta})\equiv g(\theta)$. Then for $\theta\in (a,b)$ the characteristic function of $T_{\theta}$ exists and is given by 
	$$\psi_{T_{\theta}}(t)=\exp[G(\theta +it)-G(\theta)],\,\,\,\,\,t\in (a-\theta,b-\theta),$$ where $G(z)$ is the analytic extension to the complex plane of $\int g(w)dw$. 
\end{theorem}

We are ready to state the main result of this section.

\begin{theorem}\label{theo_conv} Let $N_{\lambda}\sim MP(\lambda,W_{\phi})$, $\{X_n\}_{n\in \mathbb{N}}$ a sequence of i.i.d. random variables with $E(X_1)=\mu\in\mathbb{R}$ and $\mbox{Var}(X_1)=\sigma^2>0$. There exist numbers $a_{\lambda} = \frac{1}{\sqrt{\lambda}}$ and $d_{\lambda} = \mu\left(\frac{1}{\sqrt{\lambda}} - 1\right)$ such that 
\begin{equation*}
 \widetilde{S}_{\lambda} = a_{\lambda}\sum\limits_{i = 1}^{N_{\lambda}}(X_i + d_{\lambda}) \xrightarrow[\lambda\rightarrow\infty]{d} Y,
\end{equation*}
where $Y$ is a random variable with $ch.f.$
\begin{eqnarray}\label{chfY}
 \psi_Y(t) = \exp\left[-\phi\left\{b(\xi_0) - b\left[\xi_0 + \phi^{-1}\left(it\mu - \frac{t^2\sigma^2}{2} \right)\right]\right\} \right].
\end{eqnarray}
\end{theorem}

\begin{proof}

First note that (\ref{EF_density}) can be written in the form of (\ref{EF_density2}) by taking $\theta=\phi\xi_0$. It follows from Theorem \ref{sampson} that the $ch.f.$ of $W_\phi$ is given by
\begin{equation}\label{FEchf}
\psi_{W_\phi}(t) = \exp\left\{-\phi\left[b(\xi_0) - b\left(\xi_0 + \frac{it}{\phi}\right)\right]\right\},\,\,\,t\in\mathbb{R}.
\end{equation}

From (\ref{FEchf}), we immediately obtain that
\begin{eqnarray}\label{chf_nw}
\psi_{N_{\lambda}}(t) = \psi_{W_{\phi}}[\lambda(e^{it} - 1)]= \exp\left\{-\phi\left[b(\xi_0) - b\left(\xi_0 + \frac{i}{\phi}\lambda\left(e^{it} - 1\right)\right)\right]\right\},\quad t\in\mathbb R.
\end{eqnarray}

The Tower Property of conditional expectations gives
$$\psi_{\widetilde{S}_{\lambda}}(t) = E\left[E\left(\exp\left\{it\left(a_{\lambda}\sum\limits_{i = 1}^{N_{\lambda}}(X_i + d_{\lambda})\right)\right\} \bigg| N_{\lambda}\right)\right].$$ 

Let $G_{N_{\lambda}}(\cdot)$ denote the probability generating function of $N_{\lambda}$. Then, we use (\ref{chf_nw}) to obtain 
\begin{eqnarray*}
    \psi_{\widetilde{S}_{\lambda}}(t) &=& G_{N_{\lambda}}\left(\psi_{X_1 - \mu}\left( \frac{t}{\sqrt{\lambda}}\right)e^{i\frac{\mu}{\lambda}t}\right)= \psi_{N_{\lambda}}\left(\frac{1}{i}\log\left\{\psi_{X_1 - \mu}\left( \frac{t}{\sqrt{\lambda}}\right)e^{i\frac{\mu}{\lambda}t}\right\}\right)\\
    &=& \exp\left[- \phi \left\{b(\xi_0) - b\left[\xi_0 + \frac{i\lambda}{\phi}\left(\psi_{X_1 - \mu}\left(\frac{t}{\sqrt{\lambda}} \right)e^{i\frac{\mu}{\lambda}t}  - 1\right)  \right] \right\} \right].
\end{eqnarray*}

Taking $\lambda\rightarrow\infty$ and applying L'H\^ opital's rule twice (in the second-order derivative we are using the assumption of finite variance of of the sequence $\{X_n\}$) we obtain
\begin{equation*}\label{sem nome}
   \lim\limits_{\lambda \to \infty} \psi_{\widetilde{S}_{\lambda}}(t) = \exp\left[-\phi\left\{b(\xi_0) - b\left[\xi_0 + \phi^{-1}\left(it\mu - \frac{t^2\sigma^2}{2}\right)\right]\right\} \right] \equiv \psi_Y(t), \; \forall\; t \in \mathbb{R}. 
\end{equation*}
\end{proof}

A special case of Theorem \ref{theo_conv} is obtained when the sequence of random variables has null-mean. 
\begin{corollary} Let $N_{\lambda}\sim MP(\lambda,W_{\phi})$, $\{X_n\}_{n\in \mathbb{N}}$ a sequence of i.i.d. random variables with $E(X_1)=0$ and $Var(X_1)=1$. Then, 
    $\lim\limits_{\lambda \to \infty} \psi_{\widetilde{S}_{\lambda}}(t) = \psi_{W_\phi}\left(-\frac{t^2}{2}\right)$, for $t\in\mathbb R$.
\end{corollary}

Before we move on to more theoretical results, let us present a few examples.

\begin{example}\label{example_1}
If $N_{\lambda}$ has negative binomial distribution with parameters  $\lambda$ and $\phi = 1$, then  $b(\theta) = -\log{(-\theta)}, \xi_0 = -1$ and $c(w; 1) = 0$. From Theorem \ref{theo_conv}, it follows that
\begin{equation*}
 \psi_{Y}(t) = \frac{1}{1 - \frac{t^2\sigma^2}{2} - t\mu},\quad t\in\mathbb R.
\end{equation*}

This is the $ch.f.$ of an asymmetric Laplace distribution with parameters $\mu \in \mathbb{R}$ and $\sigma^2 \ge 0$, denoted here by $\mbox{AL}(\mu,\sigma^2)$. In other words, $\widetilde{S}_{\lambda} \overset{d}\to AL(\mu, \sigma^2)$ as $\lambda\rightarrow\infty$.
\end{example}

In Example \ref{example_1}, we have that the density function of $Y$ can be expressed in terms of the parameterization given in \cite{kkp}, $i.e.,$
\begin{align*}
    f_Y(y)= \frac{\sqrt{2}}{\sigma}\frac{\kappa}{1 + \kappa^2}\left\{\begin{array}{rc}
\exp{\left(- \frac{\sqrt{2}\kappa}{\sigma}|y|\right)},&\mbox{for}\quad y \ge 0,\\
\exp{\left(- \frac{\sqrt{2}}{\sigma \kappa}|y|\right)}, &\mbox{for}\quad y < 0,\\
\end{array}\right.
\end{align*}
where $\kappa = \frac{\sqrt{2\sigma^2 + \mu^2} - \mu}{\sqrt{2}\sigma}$ is the skewness parameter. 

\begin{example} We say a random variable $Z$ has normal inverse-Gaussian distribution with parameters $\alpha$, $\beta$, $\gamma$ and $\delta$, and write $X\sim NIG (\alpha, \beta, \gamma, \delta)$, if its $ch.f.$ is given by
$$\psi_Z(t)=exp\left\{\delta[\sqrt{\alpha^2-\beta^2}-\sqrt{\alpha^2-(\beta+it)^2}]+\gamma it\right\}, \,\,t\in\mathbb{R}.$$

See \cite{bn} for more details on this distribution. Now, if $N_{\lambda}$ has Poisson inverse-Gaussian distribution with parameters $\lambda$ and $\phi$ ($N_{\lambda} \sim PIG(\lambda, \phi)$), then $b(\theta) = -\sqrt{(-2\theta)}$ and $\xi_0 = -\frac{1}{2}$. Using again Theorem \ref{theo_conv}, we get
$$ \psi_Y(t) = \exp\left\{\phi\left(1 - \sqrt{1 - \phi^{-1}\left(t^2\sigma^2 + 2t\mu\right)}\right)\right\}.$$

This is the $ch.f.$ of a random variable with normal inverse-Gaussian distribution with parameters $\alpha = \sqrt{\frac{\phi}{\sigma^2} + \frac{\mu^2}{\sigma^4}}$, $\beta = \frac{\mu}{\sigma^2}$, $\gamma = 0$ and $\delta = \sqrt{\phi}\sigma$. 
Therefore, $\widetilde{S}_{\lambda} \overset{d}\longrightarrow NIG\left(\sqrt{\frac{\phi}{\sigma^2} + \frac{\mu^2}{\sigma^4}}, \frac{\mu}{\sigma^2}, 0, \sqrt{\phi}\sigma \right)$ as $\lambda\rightarrow\infty$.

\end{example}

The above examples provide characterizations for the Laplace and NIG distributions as weak limits of properly normalized mixed Poisson random sums.

\subsection{Mixed Poisson-stability}

In this section we introduce the notion of a stable mixed Poisson distribution. Our aim is to characterize such a distribution in terms of its $ch.f.$.  
We start with the following definition.

\begin{definition}\label{mps_def} A random variable $Y$ is said to be mixed Poisson stable (MP-stable) with respect to the summation scheme, if there exist a sequence of $i.i.d.$ random variables $\{X_n\}_{n=1}^\infty$, a mixed Poisson random variable $N_{\lambda}$ independent of all $X_i$, and constants $a_{\lambda}>0$, $d_{\lambda}\in\mathbb{R}$  such that 
\begin{equation}\label{pm-estavel}
a_{\lambda}\sum\limits_{i = 1}^{N_{\lambda}}(X_i + d_{\lambda}) \overset{d}{\to} Y, 
\end{equation}
when $\lambda\rightarrow \infty$.
If $d_{\lambda} = 0$, we say $Y$ is strictly mixed Poisson stable.
\end{definition}

One of the most important objects in the theory of stable laws is the description of domains of attraction of stable laws. The definition of a domain of attraction is as follows.

\begin{definition} Let $\{X_i\}_{i\in \mathbb{N}}$ be a sequence of $i.i.d.$ random variables with distribution function $F$ and let $\{S_n,\,n\geq 1\}$ be the partial sums. We say that $F$ belongs to the domain of attraction of a (non-degenerate) distribution $G$ if there exist normalizing sequences $\{a_n\}_{n\in\mathbb{N}}$ ($a_n>0$) and $\{d_n\}_{n\in\mathbb{N}}$ such that
\begin{equation*}
\frac{S_n-d_n}{a_n}\xrightarrow[n\rightarrow\infty]{d}G.
\end{equation*} 
We denote $F\in\mathcal{D}(G)$.
\end{definition}

It turns out that $G$ possesses a domain of attraction if, and only if, $G$ is a stable distribution (see Theorem 3.1, Chapter 9 in \cite{g}).
The following theorem gives a characterization of MP-stable distributions in terms of its $ch.f.$.  

\begin{theorem}\label{stable_th}
Assume the sequence $\{X_n\}_{n=1}^\infty$ is according to Definition \ref{mps_def} and that its distribution function $F$ satisfies $F\in\mathcal{D}(G)$ for some $\alpha$-stable distribution $G$. Then, $Y$ is PM-stable if and only if its $ch.f.$ $\psi_Y$ is of the form
\begin{eqnarray}\label{st_ch}
\psi_Y(t) = \exp\left\{-\phi\left[b(\xi_0) - b\left(\xi_0 + \frac{1}{\phi}\log{\Psi(t)}\right)\right]\right\},
\end{eqnarray}
where $\Psi(t)$ is the $ch.f.$ of some $\alpha$-stable distribution.
\end{theorem}

\begin{proof}
By Lévy's Continuity Theorem, the convergence in Expression (\ref{pm-estavel}) holds if and only if
\begin{eqnarray}\label{convergencia}
\exp\left\{-\phi\left[b(\xi_0) - b\left(\xi_0 + \frac{\lambda}{\phi}\left(\varphi_{\lambda}(t) - 1\right]\right)\right)\right\} \xrightarrow[\lambda\rightarrow \infty]{} \psi_Y(t),
\end{eqnarray}
where $\varphi_{\lambda}(t) = \Psi_{X_1}(a_{\lambda}t)\exp(ita_{\lambda}d_{\lambda})$. 
Since $b(\cdot)$ is invertible (which follows by the continuity assumption), $b'(x)>0$ (so the function is monotone increasing) and $supp(EF(\phi))\subset\mathbb{R}^+$), it follows that (\ref{convergencia}) is equivalent to
\begin{equation*}
\lambda(\varphi_{\lambda}(t)-1) \xrightarrow[\lambda\rightarrow\infty]{}\phi\left\{b^{-1}\left(b(\xi_0) + \frac{1}{\phi}\log{\psi_Y(t)}\right) - \xi_0\right\},\,\,\,\forall t\in\mathbb{R}.
\end{equation*}

We can take the limit above with $\lambda\equiv \lambda_n=n\in\mathbb{N}$ instead of $\lambda\in\mathbb R$; $\{\lambda_n\}$ can be seen as a subsequence. In this case, letting $a_{\lambda}=a_n$ and $d_{\lambda}=d_n$, it follows that 
\begin{equation*}
n(\varphi_{n}(t)-1) \xrightarrow[n\rightarrow\infty]{}\phi\left\{b^{-1}\left(b(\xi_0) + \frac{1}{\phi}\log{\psi_Y(t)}\right) - \xi_0\right\},\,\,\,\forall t\in\mathbb{R}.
\end{equation*}

From Theorem 1 in Chapter XVII in \cite{f}, the above limit implies that
\begin{equation}\label{feller}
(\varphi_{n}(t))^n \xrightarrow[n\rightarrow\infty]{}\Psi(t),\,\,\,\forall t\in\mathbb{R},
\end{equation}
where
\begin{equation*}
\Psi(t)=\exp\left\{\phi\left\{b^{-1}\left[b(\xi_0) + \frac{1}{\phi}\log{\psi_Y(t)}\right] - \xi_0\right\}\right\},\,\,\,\forall t\in\mathbb{R}.
\end{equation*}

Since the $ch.f.$ on the left hand side of $(\ref{feller})$ is the $ch.f.$ of 
$$a_{n}\sum\limits_{i = 1}^n(X_i + d_n)
,$$ and $F\in\mathcal{D}(G)$, it follows  that $\Psi(t)$ is the $ch.f.$ of some $\alpha$-stable distribution (see Chapter 9 in  \cite{g} for example). Since 
\begin{eqnarray*}
\psi_Y(t) = \exp\left\{-\phi\left[b(\xi_0) - b\left(\xi_0 + \frac{1}{\phi}\log{\Psi(t)}\right)\right]\right\},
\end{eqnarray*}
we have the sufficiency part of the theorem.

Conversely, if (\ref{st_ch}) holds, then
$\Psi(t)$ is the $ch.f.$ of some $\alpha$-stable distribution. Therefore, there exist a random variable $Z$, an $i.i.d.$ sequence $\{X_i\}_{i\in \mathbb{N}}$ of random variables and real sequences $\{a_n\}_{n\in\mathbb{N}}$ and $\{d_n\}_{n\in\mathbb{N}}$ such that 
\begin{equation}\label{sum}
a_{n}\sum\limits_{i = 1}^n(X_i + d_n)\xrightarrow[n\rightarrow\infty]{d}Z.
\end{equation}

Denote the {\it ch.f.} of $Z$ and $X_1$ by $\Psi(t)$ and $\Psi_{X_1}(t)$, respectively. Let $\gamma_{n}(t) = \Psi_{X_1}(a_n t)\exp(ita_n d_n)$,  $n\in\mathbb N$. Then, the weak limit in (\ref{sum}) is equivalent to 
\begin{equation*}
(\gamma_{n}(t))^n \xrightarrow[n\rightarrow\infty]{}\Psi(t),\,\,\,\forall t\in\mathbb{R}.
\end{equation*}

From \cite{f}, we have that the above limit implies that 
$$n(\gamma_{n}(t)-1) \xrightarrow[n\rightarrow\infty]{}\log\Psi(t).$$

Since by hypothesis
$$\Psi(t) = \exp\left\{\phi\left\{b^{-1}\left(b(\xi_0) + \frac{1}{\phi}\log{\psi_Y(t)}\right) - \xi_0\right\}\right\},$$ we obtain that
$$n(\gamma_{n}(t)-1) \xrightarrow[n\rightarrow\infty]{}\phi\left\{b^{-1}\left(b(\xi_0) + \frac{1}{\phi}\log{\psi_Y(t)}\right) - \xi_0\right\},$$ which is equivalent to 
$$\lambda(\gamma_{\lambda}(t)-1) \xrightarrow[\lambda\rightarrow\infty]{}\phi\left\{b^{-1}\left(b(\xi_0) + \frac{1}{\phi}\log{\psi_Y(t)}\right) - \xi_0\right\}.$$

The above limit gives Equation (\ref{convergencia}) with $\gamma_{\lambda}(t)$ instead of $\varphi_{\lambda}(t)$. This completes the proof of the desired result.
\end{proof}

We now apply Theorem \ref{stable_th} to three cases of interest.

\begin{example}\label{ex_bn} Take $N_\lambda\sim\mbox{MP}(\lambda,W_{\phi})$ with $W_{\phi} \sim Gamma(\phi)$. In this case $N_{\lambda}\sim\mbox{NB}(\lambda,\phi)$ with probability function 
$$P(N_{\lambda}=n) = \frac{\Gamma(n + \phi)}{n! \Gamma(\phi)}\left(\frac{\lambda}{\lambda + \phi}\right)^n\left(\frac{\phi}{\lambda + \phi}\right)^{\phi}, \; \;  n = 0, 1, \cdots.$$ Also, we have $b(\theta) = -\log{(-\theta)}$ and $\xi_0 = -1$. Apply Theorem \ref{stable_th} to deduce that a random variable $Y$ is NB-stable if and only if 
\begin{eqnarray}\label{bn-stable}
\psi_Y(t) = \left\{1 - \phi^{-1}\log{\Psi(t)}\right\}^{-\phi},
\end{eqnarray}
where $\Psi(t)$ is the $ch.f.$ of some $\alpha$-stable distribution.
\end{example}

We emphasize that Theorem \ref{stable_th} generalizes Proposition 1 in \cite{mr}. To obtain their result it is enough to take $\phi=1$ in Example \ref{ex_bn}.

\begin{example}
Let $\Psi(t) = e^{-c|t|^{\alpha}}$, $t\in \mathbb{R}$, be the $ch.f.$ of a symmetric $\alpha$-stable distribution. By Equation (\ref{bn-stable}),
$\psi_Y(t) = \left(1 + \frac{c}{\phi}|t|^{\alpha}\right)^{-\phi}$
is the $ch.f.$ of some NB-stable distribution. In particular, for $\alpha = 2$, we have that
$\psi_Y(t) = \left(1 + \frac{c}{\phi}t^2\right)^{-\phi}$
is the $ch.f.$ of a NB-stable random variable $Y$ with density function
\begin{eqnarray*}
f_Y(y) = \left(\frac{\phi}{c}\right)^{\frac{\phi}{2} + \frac{1}{4}} 2^{\frac{1}{2} - \phi}\frac{\mathcal{K}_{\phi - \frac{1}{2}}\left(y\sqrt{\phi/c}\right)}{\sqrt{\pi}\Gamma(\phi)}y^{\phi - \frac{1}{2}},\,\,\,y\in\mathbb{R},
\end{eqnarray*}
where $\mathcal{K}_{\nu}(z) =  \frac{\Gamma\left(\nu + \frac{1}{2}\right)(2z)^{\nu}}{\sqrt{\pi}}\displaystyle\int_{0}^{\infty}\frac{\cos{u}}{(u^2 + z^2)^{\nu + \frac{1}{2}}}du$ is the modified Bessel function of the second kind; see Chapter 9 in \cite{as}.
\end{example}

\begin{example}\label{ex_pig}
Consider $N_\lambda\sim\mbox{MP}(\lambda,W_{\phi})$ with $W_{\phi}\sim \mbox{IG}(\phi)$. In this case $N_{\lambda}\sim\mbox{PIG}(\lambda,\phi)$ with probability function 
$$P(N_{\lambda}=n) = \sqrt{\frac{2}{\pi}} \left[\phi(\phi + 2\lambda)\right]^{-(n - \frac{1}{2})}\frac{e^{\phi}(\lambda\phi)^n}{n!}\mathcal{L}_{n - \frac{1}{2}}\left(\sqrt{\phi(\phi + 2\lambda)}\right),\,\,n = 0, 1, \cdots$$ where $\mathcal{L}_{\nu}(z) = \frac{1}{2}\int_{0}^{\infty}u^{\nu - 1}\exp\left\{-\frac{1}{2}z(u + u^{-1})\right\}du$ is the modified Bessel function of the third kind; see \cite{as}. Also, we have $b(\theta) = -\sqrt{-2\theta}$, $\xi_0 = -\frac{1}{2}$. Applying Theorem \ref{stable_th} we obtain that $Y$ is PIG-stable if and only if 
$\psi_Y(t) = \exp\left\{\phi\left(1 - \sqrt{1 - \frac{2}{\phi}\log{\Psi(t)}}\right)\right\}$,
where $\Psi(t)$ is the $ch.f.$ of some $\alpha$-stable distribution.

For instance, by taking the $ch.f.$ of a normal distribution with mean $\mu$ and variance $\sigma^2$, that is $\Psi(t) = e^{i\mu t -\frac{1}{2}\sigma^2t^2}$, it follows that
$\psi_Y(t) = \exp\left\{\phi\left(1 - \sqrt{1 - \phi^{-1}\left(2i\mu t - \sigma^2t^2\right)}\right)\right\}$,
which is the $ch.f.$ of the $\mbox{NIG}\left(\sqrt{\frac{\phi}{\sigma^2} + \frac{\mu^2}{\sigma^4}}, \frac{\mu}{\sigma^2}, 0, \sqrt{\phi\sigma^2}\right)$ distribution. In other words, the normal inverse-Gaussian distribution is PIG-stable.
\end{example}

\section{Properties of the limiting distribution}\label{lim_distributions}

In this section we obtain statistical properties of the limiting class of distributions arising from Theorem \ref{theo_conv} with $ch.f.$ (\ref{chfY}). The main result here is the stochastic representation of these distributions as a normal mean-variance mixture \cite{baretal1982} with latent effect belonging to an exponential family. We emphasize that this class of normal exponential family (NEF) mixture distributions is new in the literature.

\begin{proposition}\label{mixture}
 Let $Y$ be a random variable with $ch.f.$ (\ref{chfY}). Then $Y$ satisfies the following stochastic representation:
	\begin{equation*}\label{geral}
	Y \overset{d}= \mu W_{\phi} + \sigma\sqrt{W_{\phi}}Z,
	\end{equation*}
where $W_{\phi} \sim\mbox{EF}(\phi)$ and $Z \sim N(0,1)$ are independent and `$\overset{d}=$' stands for equality in distribution.
\end{proposition}

\begin{proof}
	By standard properties of condicional expectation and using the $ch.f.$ of $W_{\phi}$ given in (\ref{FEchf}), we obtain that	
	\begin{eqnarray*}
		\Psi_{\mu W_{\phi} + \sigma \sqrt{W_{\phi}}Z}(t) &=& E\left(\exp\left\{it(\mu W_{\phi} + \sigma \sqrt{W_{\phi}}Z)\right\}\right)  = E\left[E\left(\exp\{it(\mu W_{\phi} + \sigma \sqrt{W_{\phi}}Z)\}| W_{\phi}\right)\right] \\
		&=& \int\limits_{0}^{\infty}e^{it\mu w}E\left(e^{it\sigma\sqrt{w}Z}\right)f_{W_\phi}(w)dw
		= \int\limits_{0}^{\infty}\exp\left\{w\left(it\mu-\frac{1}{2}t^2\sigma^2\right)\right\}f_{W_\phi}(w)dw\\
		&=& \psi_{W_\phi}\left(it\mu - \frac{t^2\sigma^2}{2}\right) 
		= \exp\left\{-\phi\left[b(\xi_0) - b\left(\xi_0 + \frac{1}{\phi}\left(it\mu - \frac{1}{2}t^2\sigma^2 \right) \right)\right]\right\}, 
	\end{eqnarray*}
which is the characteristic function given in Proposition \ref{theo_conv}. 
\end{proof}

Since we rely on an Expectation-Maximization algorithm to estimate the parameters of the class of normal mean-variance mixture distributions, the stochastic representation given in the previous proposition plays an important role in this paper. Furthermore, this representation enables us to find explicit forms for the corresponding density function as stated in the following proposition (which proof follows directly and therefore it is omitted) and examples.

\begin{proposition}\label{general density}
Let $Y = \mu W_{\phi} + \sigma\sqrt{W_{\phi}}Z$ with $Z \sim N(0,1)$ and $W_{\phi}$ independent, $\mu \in \mathbb{R}$, $\sigma^2 > 0$ and $\phi > 0$. The density function of $Y$ is given by

\begin{align}\label{density}
f_Y(y) \!=\! \frac{e^{\frac{y\mu}{\sigma^2} \!-\! \phi b(\xi_0) \!+\! d(\phi)}}{\sqrt{2\pi\sigma^2}}\int_{0}^{\infty}\!
     e^{\phi g(w)+h(w)}\!
      w^{-\frac{1}{2}}\! \exp\left\{-\frac{1}{2}\left[\left(\frac{\mu^2}{\sigma^2}-2\phi\xi_0\right)w + \frac{y^2}{\sigma^2}\frac{1}{w}\right]\right\}dw,
\end{align}
 for $y\in\mathbb{R}$.
\end{proposition}

\begin{example}\label{gamma mixed density}
Let $Y = \mu W_{\phi} + \sigma\sqrt{W_{\phi}}Z$ with $W_{\phi} \sim\mbox{Gamma}(\phi)$ independent of $Z \sim N(0,1)$, $\mu \in \mathbb{R}$, $\sigma^2 > 0$ and $\phi > 0$. In this case, we use $d(\phi) = \phi\log\phi - \log\Gamma(\phi)$, $g(w) = \log w$ and $h(w) = -\log{w}$ in Equation (\ref{density}) to obtain
\begin{eqnarray}\label{normal gamma density}
 f_Y(y) = \sqrt{\frac{2}{\pi\sigma^2}}\frac{\phi^{\phi}}{\Gamma(\phi)}e^{y\mu/\sigma^2}\mathcal{K}_{\phi - \frac{1}{2}}\left(\sqrt{\left(\frac{\mu^2}{\sigma^2} + 2\phi\right)\left(\frac{y^2}{\sigma^2}\right)}\right)\left(\frac{\frac{y^2}{\sigma^2}}{\frac{\mu^2}{\sigma^2} + 2\phi}\right)^{\frac{\phi}{2} - \frac{1}{4}},\, y\in\mathbb{R},
\end{eqnarray}
where $\mathcal{K}_{\cdot}(\cdot)$ is the modified Bessel function of the third kind. This function satisfies the property $\mathcal{K}_{\frac{1}{2}}(z) =\left(\frac{\pi}{2z}\right)^{\frac{1}{2}}e^{-z}$, for $z\in\mathbb{R}$. Using this fact and replacing $\phi = 1$ in Equation (\ref{normal gamma density}), we obtain the probability density function of the Asymmetric Laplace distribution.

\end{example}

\begin{example}\label{inverse gaussian mixed density}
If we assume $W_{\phi} \sim\mbox{IG}(\phi)$, then $d(\phi) = \frac{1}{2}\log{\phi},\, g(w) = -(2w)^{-1}$ and $h(w) = -\frac{1}{2}\log(2\pi w^3)$. Hence, it follows from Equation (\ref{density}) that
\begin{eqnarray}\label{normal inverse gaussian density}
f_Y(y)= \frac{1}{\pi}\sqrt{\frac{\phi}{\sigma^2}}\exp\left(\frac{y\mu}{\sigma^2} + \phi\right)\mathcal{K}_{-1}\left(\sqrt{\left(\frac{\mu^2}{\sigma^2} + \phi\right)\left(\frac{y^2}{\sigma^2} + \phi\right)}\right)\left(\frac{\mu^2 + \sigma^2\phi}{y^2 + \sigma^2\phi}\right)^{\frac{1}{2}},\quad y\in\mathbb R.
\end{eqnarray}
The density function (\ref{normal inverse gaussian density}) corresponds to a NIG distribution.
\end{example}

\begin{example}
Consider $W_{\phi}$ following a generalized hyperbolic secant distribution with dispersion parameter $\phi>0$; for more details on this distribution see \cite{bs}. In this case, the density function of $Y$, for $y \in \mathbb{R}$, can be expressed by
\begin{eqnarray*}
f_Y(y)=\frac{2^{\frac{\phi - 5}{2}}}{\sqrt{\pi^3\sigma^2}}\frac{\phi}{\Gamma(\phi)}\exp(\mu y/\sigma^2) E\left(\Gamma\left(\left|\phi + iU\right|^2/4\right)\right),\,y\in\mathbb{R},
\end{eqnarray*}
where  $U\sim GIG\left(\frac{3\pi\phi}{2} + \frac{\mu^2}{\sigma^2}, \frac{y^2}{\sigma^2}, \frac{1}{2}\right)$.
\end{example}

We conclude this section with a numerical illustration of the weak convergence obtained in Theorem \ref{theo_conv} through a small Monte Carlo simulation. We generate random samples (500 replicas) from the partial sums $\tilde{S}_{\lambda}$ with $\mbox{NB}(\lambda,\phi)$ and $\mbox{PIG}(\lambda,\phi)$ number of terms; we set $\phi=2$ and $\lambda=30,50,500$. The sequence $\{X_n\}_{n\geq1}$ is generated from the exponential distribution with mean equal to 1 (in this case $\mu=\sigma^2=1$). Figures \ref{graph sum mixed gamma} and \ref{graph sum mixed inverse gaussian} show the histograms of the generated random samples with the curve of the corresponding density function for the NB and PIG cases, respectively. As expected, we observe a good agreement between the histograms and the theoretical densities as $\lambda$ increases, which is according Theorem \ref{theo_conv}.

\begin{figure}[H]
\centering
\subfigure[$\lambda = 30$]{\includegraphics[width=\textwidth,height=0.2\textheight,keepaspectratio]{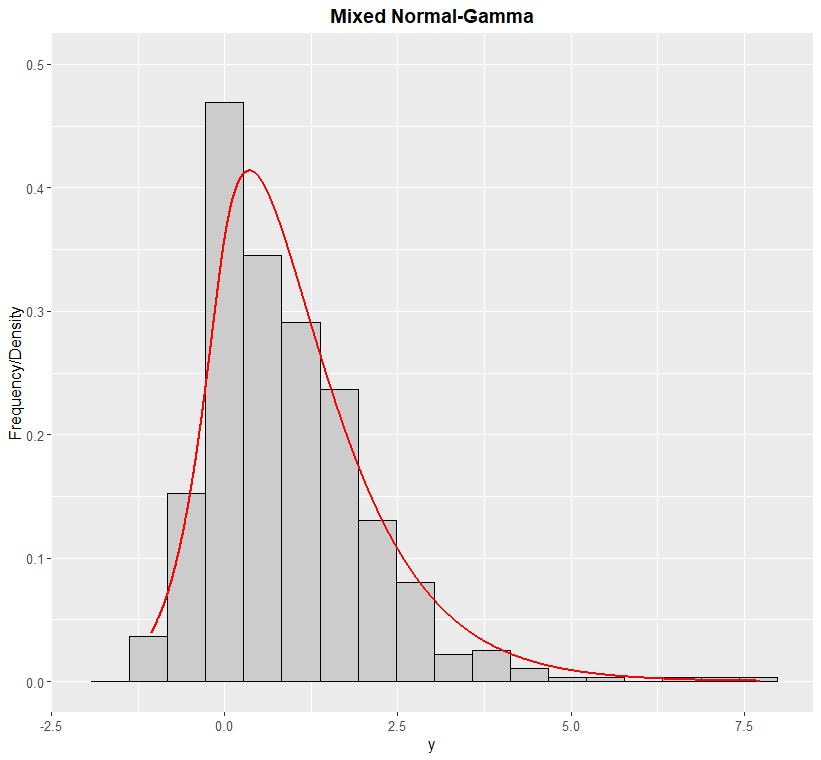}}
\subfigure[$\lambda = 50$]{\includegraphics[width=\textwidth,height=0.2\textheight,keepaspectratio]{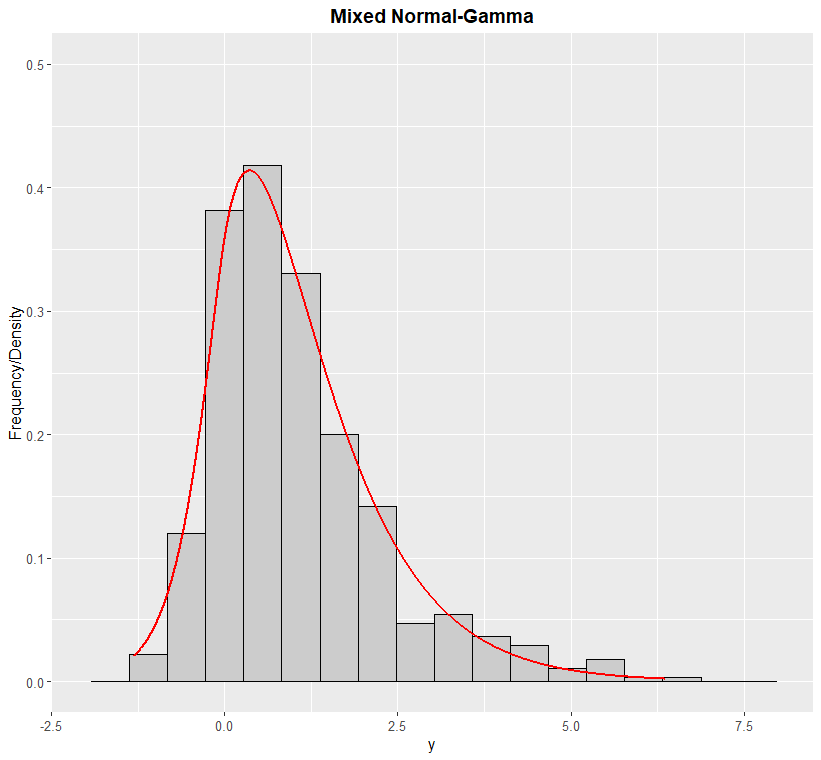}}
\subfigure[$\lambda = 500$]{\includegraphics[width=\textwidth,height=0.2\textheight,keepaspectratio]{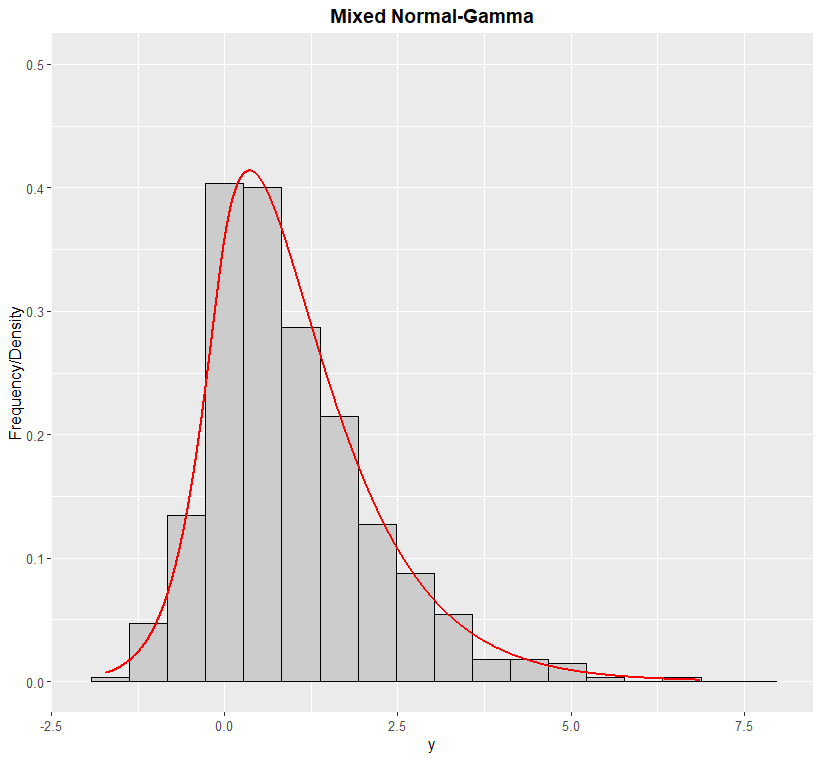}}
\caption{Histograms for the generated random sample from $\widetilde{S}_{\lambda}$ with $N_{\lambda} \sim\mbox{BN}(\lambda, 2)$ for $\lambda = 30, 50, 500$ and normal-gamma density function.}
\label{graph sum mixed gamma}
\end{figure}

\begin{figure}[H]
\centering
\subfigure[$\lambda = 30$]{\includegraphics[width=\textwidth,height=0.2\textheight,keepaspectratio]{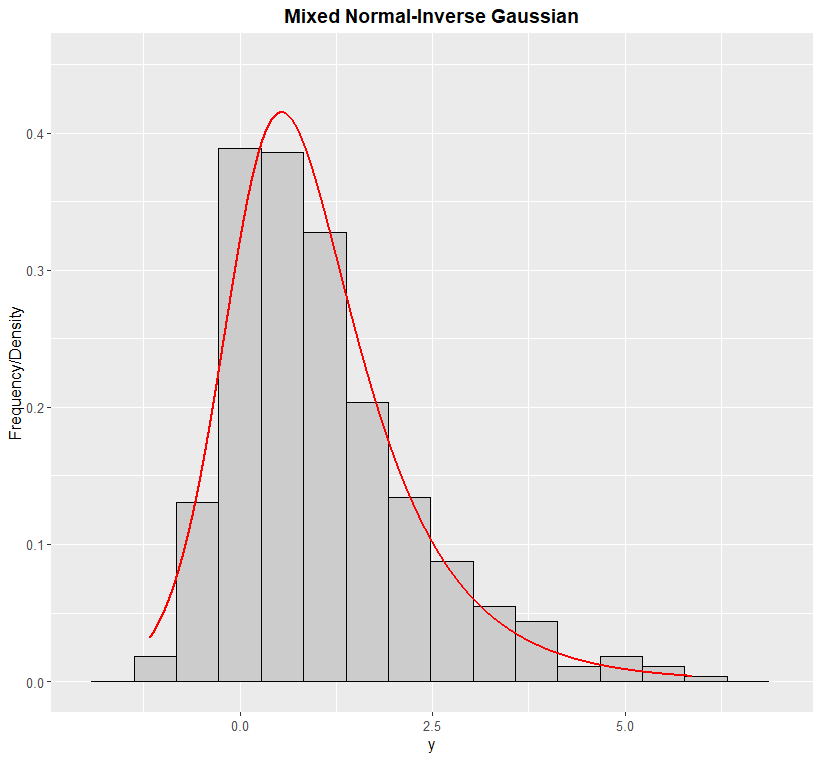}}
\subfigure[$\lambda = 50$]{\includegraphics[width=\textwidth,height=0.2\textheight,keepaspectratio]{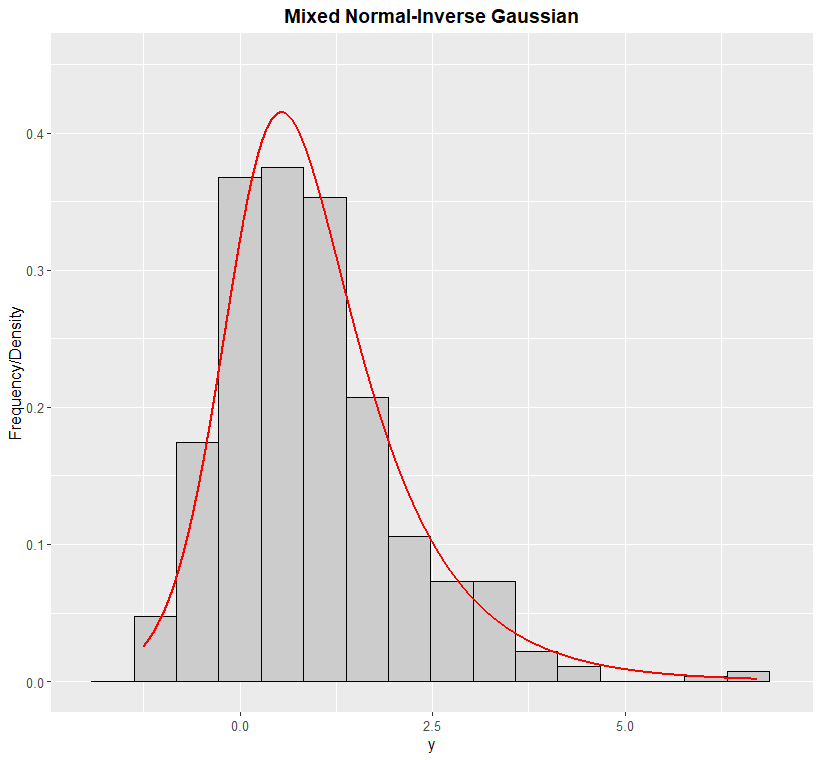}}
\subfigure[$\lambda = 500$]{\includegraphics[width=\textwidth,height=0.2\textheight,keepaspectratio]{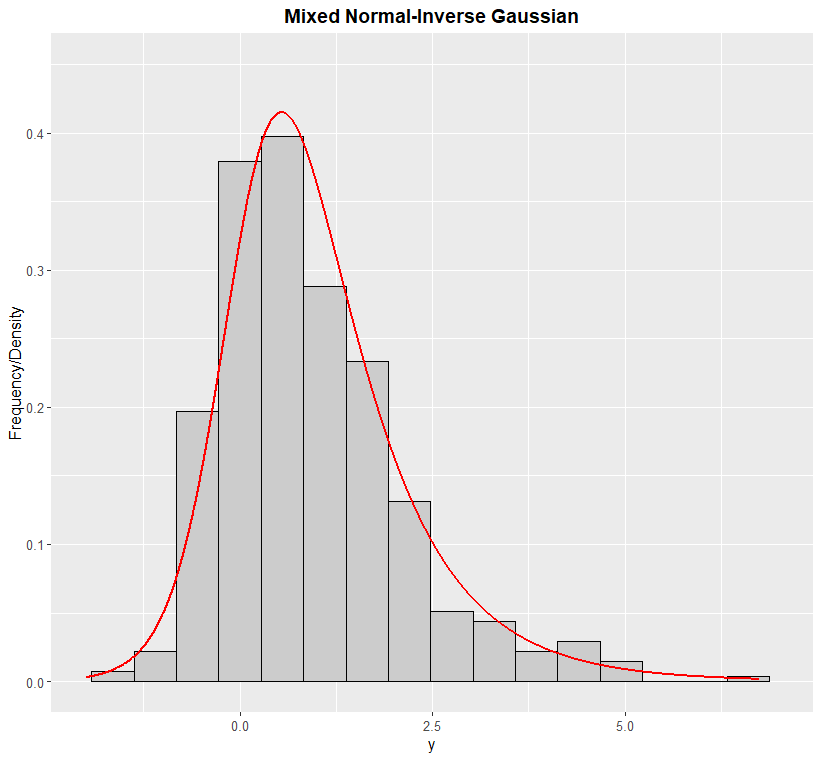}}
\caption{Histograms for the generated random sample from $\widetilde{S}_{\lambda}$ with $N_{\lambda} \sim\mbox{PIG}(\lambda, 2)$ for $\lambda = 30, 50, 500$ and normal-gamma density function.}
\label{graph sum mixed inverse gaussian}
\end{figure}

\section{Inference for NEF laws}\label{inference}

In this section we discuss estimation of the parameters of the limiting class of normal-exponential family laws obtained in Section \ref{theo_results}. We consider the method of moments and maximum likelihood estimation via Expectation-Maximization algorithm. Throughout this section, ${\bf Y}=(Y_1,\ldots,Y_n)^\top$ denotes a random sample ({\it i.i.d.}) from the NEF distribution and $n$ stands for the sample size.

\subsection{Method of moments}\label{MM}

Let $Y\sim\mbox{NEF}(\mu,\sigma^2,\phi)$. By using the characteristic function of the NEF distributions given in Proposition \ref{theo_conv}, we see that the first four cumulants of $Y$ are given by
\begin{equation}\label{k4}
    \begin{aligned}
 \kappa_1 &\equiv E(Y) = \mu,\\ 
 \kappa_2 &\equiv\mbox{Var}(Y) = \dfrac{\mu^2b''(\xi_0) + \phi\sigma^2 }{\phi},\\
 \kappa_3 &\equiv E\left((Y-\mu)^3\right)= \dfrac{\mu^3b^{(3)}(\xi_0) + 3\phi\sigma^2\mu b''(\xi_0)}{\phi^2},\\
 \kappa_4 &\equiv E\left((Y-\mu)^4\right)= \dfrac{\mu^4b^{(4)}(\xi_0) + 6\phi\sigma^2\mu^2b^{(3)}(\xi_0) + 3\phi^2\sigma^4b''(\xi_0)}{\phi^3},
 \end{aligned}
\end{equation}
where $b^{(3)}(\cdot)$ and $b^{(4)}(\cdot)$ are the third and forth derivatives of the function $b(\cdot)$. The skewness coefficient and the excess of kurtosis, denoted respectively by $\beta_1$ and $\beta_2$, can be obtained from the well-known relationships 
\begin{eqnarray}\label{kstan}
\beta_1  = \frac{\kappa_3}{\kappa_2^{3/2}} \quad \textrm{and} \quad \beta_2  = \frac{\kappa_4}{\kappa_2^2} - 3.
\end{eqnarray}

The following examples give explicit expressions for the first four cumulants for some special cases of the NEF class of distributions.

\begin{example}
Consider $Y$ following a NEF distribution with $W_{\phi} \sim \textrm{Gamma} (\phi)$. In this case $Y$ follows a normal-gamma distribution. Also, $b(\theta) = -\log(-\theta)$ for $\theta<0$ and $\xi_0 = -1$. By using these quantities and taking the derivatives of $b(\cdot)$ in (\ref{k4}) and (\ref{kstan}) we obtain that
$$\begin{cases}
 &\kappa_1 =  \mu,  \smallskip \\[0.4cm]  
 &\kappa_2 =  \dfrac{\mu^2 + \phi\sigma^2}{\phi},  \smallskip \\[0.4cm]  
 &\kappa_3 =\dfrac{2\mu^3 + 3\mu\phi\sigma^2}{\phi^2},  \smallskip \\[0.4cm]  
 &\kappa_4 =\dfrac{6\mu^4 + 12\mu^2\phi\sigma^2 + 3\phi^2\sigma^4}{\phi^3},
\end{cases}
\; \; \; \; 
\begin{cases}
 &\beta_1 = \dfrac{2\mu^3 + 3\mu\phi\sigma^2}{\sqrt{\phi}(\mu^2 + \phi\sigma^2)^{\frac{3}{2}}},  \medskip \\[0.4cm]  
 &\beta_2 = \dfrac{6\mu^4 + 12\mu^2\phi\sigma^2 + 3\phi^2\sigma^4}{\phi(\mu^2 + \phi\sigma^2)^2}.
\end{cases}
$$
\end{example}

\begin{example}
Now consider $W_{\phi} \sim \textrm{IG}(\phi)$. Then, $b(\theta) = -\sqrt{-2\theta}$ for $\theta<0$ and $\xi_0 = -1/2$. We have that $Y$ follows a NIG distribution with parameters $\mu$, $\sigma^2$ and $\phi$. Its central moments and cumulants are given by
$$\begin{cases}
 &\kappa_1 =  \mu,  \smallskip \\[0.4cm]  
 &\kappa_2 =  \dfrac{\mu^2 + \phi\sigma^2}{\phi},  \smallskip \\[0.4cm]  
 &\kappa_3 = \dfrac{3\mu^3 + 3\mu\phi\sigma^2}{\phi^2},  \smallskip \\[0.4cm]    
 &\kappa_4 = \dfrac{15\mu^4 + 18\mu^2\phi\sigma^2 + 3\phi^2\sigma^4}{\phi^3},
\end{cases}
\; \; \; \; 
\begin{cases}
 &\beta_1 = \dfrac{3\mu^3 + 3\mu\phi\sigma^2}{\sqrt{\phi}(\mu^2 + \phi\sigma^2)^{\frac{3}{2}}},  \medskip \\[0.4cm]    
 &\beta_2 = \dfrac{15\mu^4 + 18\mu^2\phi\sigma^2 + 3\phi^2\sigma^4}{\phi(\mu^2 + \phi\sigma^2)^2}.
\end{cases}
$$
\end{example}

\begin{example}
For the case $W_{\phi} \sim\mbox{GHS}(\phi)$, it follows that $b(\theta) = \frac{1}{2}\log\left(1 + \tan^2{\theta}\right)$ for $\theta\in\mathbb R$ and $\xi_0 =  -3\pi/4$. We have that $Y$ follows a normal-generalized hyperbolic secant distribution. Its central moments and cumulants are given by
$$\begin{cases}
 &\kappa_1 =  \mu,  \smallskip \\[0.4cm]  
 &\kappa_2 =  \dfrac{2\mu^2 + \sigma^2\phi}{\phi},  \smallskip \\[0.4cm]  
 &\kappa_3 = \dfrac{4\mu^3 + 6\mu\sigma^2\phi}{\phi^2},  \smallskip \\[0.4cm]    
 &\kappa_4 = \dfrac{16\mu^4 + 24\sigma^2\phi + 6\sigma^4\phi^2}{\phi^3},
\end{cases}
\; \; \; \; 
\begin{cases}
 &\beta_1 = \dfrac{4\mu^3 + 6\mu\sigma^2\phi}{\sqrt{\phi}(2\mu^2 + \sigma^2\phi)^{\frac{3}{2}}},  \medskip \\[0.4cm]    
 &\beta_2 = \dfrac{16\mu^4 + 24\sigma^2\phi + 6\sigma^4\phi^2}{\phi(2\mu^2 + \sigma^2\phi)^2}.
\end{cases}
$$
\end{example}

Let us discuss the estimation procedure. Since we have three parameters, we need three equations to estimate them. We use the three first moments $\mu_k \equiv E(Y^k)$ and its respective empirical quantities $M_k \equiv \dfrac{1}{n}\sum\limits_{i = 1}^n Y_i^k$ to do this job, where $k=1,2,3$ and $\mu_1\equiv \mu$. The theoretical moments can be obtained from the cumulants by using the relationships $\mu_1 = \kappa_1$, $\mu_2 = \kappa_2 + \mu_1^2$ and $\mu_3  = \kappa_3 + 3\mu_1\mu_2 - 2\mu_1^3$.

By equating theoretical moments with their empirical quantities, the method of moments (MM) estimators are obtained as the solution of the following system of non-linear equations: 
$$\left\{\begin{array}{rc}
&\widetilde\mu_1 =M_1\\ \noalign{\smallskip}
&\widetilde\mu_2 =M_2\\ \noalign{\smallskip}
&\widetilde\mu_3 =M_3
\end{array}\right. \Longrightarrow \left\{\begin{array}{ll}
\widetilde\mu   = M_1\\ \noalign{\smallskip}
\dfrac{\widetilde\mu^2b''(\xi_0) + \widetilde\phi\widetilde\sigma^2 }{\widetilde\phi} + \mu_1^2 = M_2\\ \noalign{\smallskip}
\dfrac{\widetilde\mu^3b^{(3)}(\xi_0) + 3\widetilde\phi\widetilde\sigma^2\widetilde\mu b''(\xi_0)}{\widetilde\phi^2} + 3\widetilde\mu_1\widetilde\mu_2 - 2\widetilde\mu_1^3 = M_3
\end{array}\right.$$

The solution of the above system of equations, denoted by $\widetilde\mu$, $\widetilde\sigma^2$ and $\widetilde\phi$, is the MM estimator and is given explicitly by
\begin{empheq}[left=\empheqlbrace]{align}
\nonumber \widetilde{\mu} &=M_1,\\ \noalign{\smallskip}
\nonumber\widetilde\sigma^2 &=M_2 - M_1^2\left(1 + \dfrac{b''(\xi_0)}{\widetilde{\phi}}\right), 
\end{empheq}
where $\widetilde{\phi}$ is the admissible solution of the quadratic equation
\begin{align}
\nonumber    \left(3M_1M_2 - 2M_1^3 - M_3\right){\widetilde\phi}^2 +
    b''(\xi_0)(3M_1M_2 - 3M_1^3)\widetilde\phi +
    M_1^3(b^{(3)}(\xi_0) - 3b''(\xi_0)^2) = 0.
\end{align}

A potential problem of the MM estimators is that estimates can lie outside of the parameter space, specially under small sample sizes. When admissible MM estimates are available, they also can be used as initial guesses for the EM-algorithm, as discussed in the sequence.

\subsection{Expectation-Maximization algorithm}\label{secaoEM}

In this section we obtain the Expectation-Maximization (EM) algorithm to find the maximum likelihood estimators for the parameters of the model. From the stochastic representation of the NEF laws, we can use $W_{\phi}$ as the latent variable to construct such estimation algorithm. 

Consider the complete data $(Y_1, W_{\phi 1}), \cdots, (Y_n, W_{\phi n})$, where $Y_1,\ldots,Y_n$ are observable variables with respective latent effects $W_{\phi 1},\ldots,W_{\phi n}$. Let $\Psi = (\mu, \sigma^2, \phi)^\top$ be the parameter vector.

The complete log-likelihood function is $\ell_c(\Psi) = \sum \limits_{i = 1}^n\log\{P(Y_i = y_i|W_{\phi i} = w_i)f_{W_{\phi}}(w_i)\}$, where $f_{W_{\phi}}(\cdot)$ is the density function of the exponential family given in Equation (\ref{EF_density}). From now on, we assume that the function $c(\cdot,\cdot)$ can be expressed as $c(w;\phi)=d(\phi)+\phi g(w)+h(w)$, with $d(\cdot)$ a three times differentiable function (see \cite{bs}). For the gamma case, we have that $d(\phi)=\phi\log\phi-\log\Gamma(\phi)$, $g(w)=\log w$ and $h(w)=-\log w$. By assuming $W_{\phi}\sim\mbox{IG}(\phi)$, we get $d(\phi)=\frac{1}{2}\log\phi$, $g(w)=-\frac{1}{2w}$ and $h(w)=-\frac{1}{2}\log(2\pi w^3)$ as well.

More explicitly, we obtain that the complete log-likelihood function takes the form
\begin{align*}
    \ell_c(\Psi) \propto \sum \limits_{i = 1}^n \left\{-\frac{1}{2}\log\sigma^2 - \frac{1}{2\sigma^2}\frac{y_i^2}{w_i} + \frac{\mu}{\sigma^2}y_i - \frac{\mu^2}{2\sigma^2}w_i + d(\phi) + \phi\left[w_i\xi_0 - b(\xi_0) + g(w_i)\right]\right\}.
\end{align*}

We now obtain the E-step and M-step of the EM algorithm with details. We denote by $\Psi^{(r)}$ the estimate of the parameter vector $\Psi$ in the $r$th loop of the EM-algorithm.

{\flushleft \textbf{E-step.}} Here, we need to find the conditional expectation of $\ell_c(\Psi)$ given the observable random variables $Y=(Y_1,\ldots,Y_n)^\top$. We denote this conditinal expectation by $Q$, which assumes the form
\begin{align*}
    \nonumber & Q(\Psi; \Psi^{(r)}) \equiv E(\ell_c(\Psi)|Y=y; \Psi^{(r)})\\
     &\propto \sum \limits_{i = 1}^n\left\{-\frac{1}{2}\log\sigma^2 - \frac{1}{2\sigma^2}y_i^2\gamma_i^{(r)} + \frac{\mu}{\sigma^2}y_i - \frac{\mu^2}{2\sigma^2}\alpha_i^{(r)} + d(\phi) + \phi\left[\xi_0\alpha_i^{(r)} - b(\xi_0)  + \delta_i^{(r)}\right]\right\},
\end{align*}
where $\gamma_i^{(r)} \equiv E\left(W_{\phi i}^{-1}|Y_i=y_i; \Psi^{(r)}\right)$, $\alpha_i^{(r)} \equiv E\left(W_{\phi i}|Y_i=y_i; \Psi^{(r)}\right)$ and $\delta_i^{(r)} \equiv E\left(g(W_{\phi i})|Y_i=y_i; \Psi^{(r)}\right)$, for $i = 1, \ldots, n$.

In the following, we obtain the conditional expectations above for the gamma and inverse-Gaussian cases. For simplicity of notation, the index $i$ is omitted.

\begin{proposition} Assume $W_{\phi} \sim \mathrm{Gamma}(\phi)$. Then, for $K, L \in \mathbb{Z}$, we have that
\begin{eqnarray*}
E\left(W_{\phi}^Kg(W_{\phi})^L\big|Y=y\right) &=& \dfrac{\mathcal{K}_{\phi + K - \frac{1}{2}}(\sqrt{ab})}{\mathcal{K}_{\phi - \frac{1}{2}}(\sqrt{ab})} \left(\dfrac{b}{a}\right)^{\frac{K}{2}}E_U\left(g(W_{\phi})^L\right),
\end{eqnarray*}
where $U \sim GIG\left(a, b, p\right)$, $a = \frac{\mu^2}{\sigma^2} + 2\phi$, $b = \frac{y^2}{\sigma^2}$, $p = \phi + K - \frac{1}{2}$, $\mathcal{K}_{\cdot}(\cdot)$ is the modified Bessel function of the third kind and $E_U(\cdot)$ denotes expecation taken with respect to the distribution of $U$. 
\end{proposition}

\begin{proof}
We have that
\begin{eqnarray*}
E\left(W_{\phi}^Kg^L(W_{\phi})\Big|Y\right) &=& \int_{0}^{\infty} w^Kg^L(w) \dfrac{f_{Y|W}(y|w)f_W(w)}{f_Y(y)}dw\\
             &=& \dfrac{1}{f_Y(y)} \int_{0}^{\infty}  \frac{w^Kg(w)^L}{\sqrt{2\pi\sigma^2w}}e^{-\frac{(y-\mu w)^2}{2\sigma^2w} + \phi[-w + \log(\phi) + \log(w)] - \log\Gamma(\phi)-\log w}dw.
\end{eqnarray*}

By using the explicit form of the normal-gamma distribution density given in Expression (\ref{normal gamma density}), we get
\begin{align*}
E\left(W_{\phi}^Kg^L(W_{\phi})\Big|Y\right)& =  \dfrac{\left(\frac{\mu^2 + 2\phi\sigma^2}{y^2}\right)^{\frac{\phi}{2} - \frac{1}{4}}}{2\mathcal{K}_{\phi - \frac{1}{2}}\left(\sqrt{\left[\frac{\mu^2}{\sigma^2}+2\phi\right]\frac{y^2}{\sigma^2}}\right)}\\
&\times\int_{0}^{\infty}g(w)^L\underbrace{w^{\left(\phi + K - \frac{1}{2}\right) - 1}\exp\left\{-\frac{1}{2}\left[\left(\frac{\mu^2}{\sigma^2} + 2\phi\right)w + \left(\frac{y^2}{\sigma^2}\right)\frac{1}{w}\right]\right\}}_{\textrm{Kernel GIG}\left(\frac{\mu^2}{\sigma^2} + 2\phi, \frac{y^2}{\sigma^2}, \phi + K - \frac{1}{2}\right)}dw.
\end{align*}

Denoting $U \sim GIG\left(a, b, p\right)$ with $a = \frac{\mu^2}{\sigma^2} + 2\phi$, $b = \frac{y^2}{\sigma^2}$ and $p = \phi + K - \frac{1}{2}$ and noting that the integrand above is the kernel of a GIG density function, we obtain the desired result.
\end{proof}

\begin{example}\label{espCondGamma}
Let $Y = \mu W_{\phi} + \sigma\sqrt{W_{\phi}}Z$ with $Z \sim N(0,1)$ and $W_{\phi} \sim \textrm{Gamma}(\phi)$ independent of each other, $\mu \in \mathbb{R}$, $\sigma^2 > 0$ and $\phi > 0$. Replacing $(K,L)=(1,0)$,  $(K,L)=(-1,0)$ and $(K,L)=(0,1)$ in the previous proposition, we get
\begin{eqnarray*}
\alpha \equiv E(W_{\phi}| Y = y) = \dfrac{\mathcal{K}_{\phi + \frac{1}{2}}\left(\sqrt{\left(\frac{\mu^2}{\sigma^2} + 2\phi\right)\left(\frac{y^2}{\sigma^2}\right)}\right)}{\mathcal{K}_{\phi - \frac{1}{2}}\left(\sqrt{\left(\frac{\mu^2}{\sigma^2} + 2\phi\right)\left(\frac{y^2}{\sigma^2}\right)}\right)}\left(\dfrac{y^2}{\mu^2 + 2\phi\sigma^2}\right)^{\frac{1}{2}},
\end{eqnarray*}
\begin{eqnarray*}
\gamma \equiv E\left({W_{\phi}}^{-1}| Y = y\right) = \dfrac{\mathcal{K}_{\phi - \frac{3}{2}}\left(\sqrt{\left(\frac{\mu^2}{\sigma^2} + 2\phi\right)\left(\frac{y^2}{\sigma^2}\right)}\right)}{\mathcal{K}_{\phi - \frac{1}{2}}\left(\sqrt{\left(\frac{\mu^2}{\sigma^2} + 2\phi\right)\left(\frac{y^2}{\sigma^2}\right)}\right)}\left(\dfrac{y^2}{\mu^2 + 2\phi\sigma^2}\right)^{-\frac{1}{2}}
\end{eqnarray*}
and
\begin{eqnarray*}
\delta = E(\log W_{\phi} | Y = y) = \frac{1}{2}\log\left(\frac{\mu^2 + 2\phi\sigma^2}{y^2}\right) + \frac{\mathcal{K}'_{\phi - \frac{1}{2}}\left(\sqrt{\left(\frac{\mu^2}{\sigma^2} + 2\phi \right)\left(\frac{y^2}{\sigma^2}\right)}\right)}{\mathcal{K}_{\phi - \frac{1}{2}}\left(\sqrt{\left(\frac{\mu^2}{\sigma^2} + 2\phi \right)\left(\frac{y^2}{\sigma^2}\right)}\right)}.
\end{eqnarray*}
\end{example}

We now present explicit expressions for the conditional expectation for the NIG case.

\begin{proposition}\label{espcondicionalIG}
Consider $W_{\phi} \sim \textrm{IG}(\phi)$. Then, for $K, L \in \mathbb{Z}$, we obtain that
\begin{eqnarray*}
E\left(W_{\phi}^Kg(W_{\phi})^L\big|Y=y\right) = \dfrac{\mathcal{K}_{K - 1}(\sqrt{ab})}{\mathcal{K}_{-1}(\sqrt{ab})} \left(\dfrac{b}{a}\right)^{\frac{K}{2}}E_U\left(g(W_{\phi})^L\right), 
\end{eqnarray*}
where $U \sim\mbox{GIG}\left(a, b, p\right)$, $a = \frac{\mu^2}{\sigma^2} + \phi$, $b = \frac{y^2}{\sigma^2} + \phi$.
\end{proposition}

\begin{proof}
We have that
\begin{eqnarray*}
E\left(W_{\phi}^Kg(W_{\phi})^L\big|Y=y\right) &=& \int_{0}^{\infty} w^Kg^L(w) \dfrac{f_{Y|W}(y|w)f_W(w)}{f_Y(y)}dw\\
             &=& \dfrac{1}{f_Y(y)} \int_{0}^{\infty} w^Kg(w)^L \frac{1}{\sqrt{2\pi\sigma^2w}}e^{-\frac{(y-\mu w)^2}{2\sigma^2w}}e^{-\frac{\phi}{2}\left(w+\frac{1}{w}\right) + \frac{1}{2}\left(\log\phi - \log(2\pi w^3)\right)}dw.
\end{eqnarray*}

Let $f_Y(y)$ be the NIG density function as given in (\ref{normal inverse gaussian density}) and denote $U \sim GIG\left(a, b, p\right)$ with $a = \frac{\mu^2}{\sigma^2} + \phi$, $b = \frac{y^2}{\sigma^2} + \phi$ and $p = K - 1$. It follows that
\begin{align*}
E\left(W_{\phi}^Kg(W_{\phi})^L\big|Y=y\right) &=  \dfrac{\left(\frac{y^2 + \phi\sigma^2}{\mu^2 + \phi\sigma^2}\right)^{\frac{1}{2}}}{\mathcal{K}_{-1}\left(\sqrt{\left(\frac{\mu^2}{\sigma^2} + \phi\right)\left(\frac{y^2}{\sigma^2} + \phi\right)}\right)}\\
&\times\int_{0}^{\infty}g(w)^L\underbrace{w^{\left(K - 1\right) - 1}\exp\left\{-\frac{1}{2}\left[\left(\frac{\mu^2}{\sigma^2} + \phi\right)w + \left(\frac{y^2}{\sigma^2} + \phi\right)\frac{1}{w}\right]\right\}}_{\textrm{Kernel GIG}\left(\frac{\mu^2}{\sigma^2} + \phi, \frac{y^2}{\sigma^2} + \phi, K - 1\right)}dw\\
&= \dfrac{\mathcal{K}_{K - 1}(\sqrt{ab})}{\mathcal{K}_{-1}(\sqrt{ab})} \left(\dfrac{b}{a}\right)^{\frac{K}{2}}E_U\left(g(W_{\phi})^L\right).
\end{align*}
\end{proof}

\begin{example}
By considering the NIG case and applying Proposition \ref{espcondicionalIG} with $(K,L)=(1,0)$, $(K,L)=(-1,0)$ and $(K,L)=(0,1)$, we obtain 
\vspace{0.1cm}\\
$\alpha \equiv E(W_{\phi}| Y = y) = \dfrac{\mathcal{K}_{0}\left(\sqrt{\left(\frac{\mu^2}{\sigma^2} + \phi\right)\left(\frac{y^2}{\sigma^2} + \phi\right)}\right)}{\mathcal{K}_{-1}\left(\sqrt{\left(\frac{\mu^2}{\sigma^2} + \phi\right)\left(\frac{y^2}{\sigma^2} + \phi\right)}\right)}\left(\dfrac{y^2 + \phi\sigma^2}{\mu^2 + \phi\sigma^2}\right)^{\frac{1}{2}},$\smallskip \\[0.3cm]
$\gamma \equiv E\left({W_{\phi}}^{-1}| Y = y\right) = \dfrac{\mathcal{K}_{-2}\left(\sqrt{\left(\frac{\mu^2}{\sigma^2} + \phi\right)\left(\frac{y^2}{\sigma^2} + \phi\right)}\right)}{\mathcal{K}_{-1}\left(\sqrt{\left(\frac{\mu^2}{\sigma^2} + \phi\right)\left(\frac{y^2}{\sigma^2} + \phi\right)}\right)}\left(\dfrac{y^2 + \phi\sigma^2}{\mu^2 + \phi\sigma^2}\right)^{-\frac{1}{2}}\; \textrm{and}$\smallskip \\[0.3cm]
$\delta \equiv E\left(-\dfrac{1}{2W_{\phi}}| Y = y\right) = -\dfrac{1}{2}\gamma.$
\end{example}

{\flushleft \textbf{M-step.}} This step of the EM-algorithm consists in maximizing the function $Q\equiv Q(\Psi; \Psi^{(r)})$. The score function associated to this function is
\begin{eqnarray*}
 \frac{\partial Q}{\partial \mu} &=&\frac{1}{\sigma^2}\sum \limits_{i = 1}^n \left\{y_i - \mu\alpha_i^{(r)}\right\},\\
 \frac{\partial Q}{\partial \sigma^2}  &=& -\frac{n}{2\sigma^2} + \frac{1}{2(\sigma^2)^2}\sum \limits_{i = 1}^n \left\{ y_i^2\gamma_i^{(r)} -2\mu y_i + \mu^2\alpha_i^{(r)}\right\},\\
 \frac{\partial Q}{\partial \phi}  &=& n(d'(\phi) - b(\xi_0)) + \sum \limits_{i = 1}^n \left\{\xi_0\alpha_i^{(r)} + \delta_i^{(r)}\right\}.
\end{eqnarray*}

The estimate of $\Psi$ in the $(r+1)$th loop of the EM-algorithm is obtained as the solution of the system of equations $\partial Q(\Psi; \Psi^{(r)})/\partial \Psi = 0$. After some algebra, we get
\begin{eqnarray*}
    \mu^{(r+1)} &=&\frac{\sum \limits_{i = 1}^n y_i}{\sum \limits_{i = 1}^n \alpha_i^{(r)}}, \quad
    {\sigma^2}^{(r+1)}=\frac{1}{n}\sum \limits_{i = 1}^n \left(y_i^2\gamma_i^{(r)} - 2\mu^{(r+1)}y_i + {\mu^{(r+1)}}^2\alpha_i^{(r)}\right) \quad \mbox{and}\\
    \phi^{(r + 1)}&=&v\left( b(\xi_0) - \frac{\xi_0}{n}\sum \limits_{i = 1}^n\alpha_i^{(r)} - \frac{1}{n}\sum \limits_{i = 1}^n\delta_i^{(r)}\right),
\end{eqnarray*}
where $v(\cdot)$ is the inverse function of $d'(\cdot)$. 

We now describe brielfy how the EM-algorithm works. As initial guess for $\Psi^{(0)}$ we can take the MM estimates. Update the conditional expectations with the previous EM-estimates, denoted by $\Psi^{(r)}$, as well as the $Q$-function. Next step is to find the maximum global point of the $Q$-function, say $\Psi^{(r+1)}$, which is provided in closed form above. Check if some convergence criterion is satisfied, for instance $||\Psi^{(r+1)}-\Psi^{(r)}||/||\Psi^{(r)}||<\epsilon$, for some small $\epsilon>0$. If this criterion is satisfied, the current EM-estimate is returned. Otherwise, update the previous EM-estimate by the current one and perform the above algorithm again until convergence is achieved.

The standard error of the parameter estimates can be obtained through the observed information matrix in \cite{l}, which is given by
\begin{eqnarray}\label{infmatrix}
I(\Psi) = E\left(- \frac{\partial l_c(\Psi)^2}{\partial \Psi \partial \Psi^T}{\big|}Y\right) - E\left(\frac{\partial l_c(\Psi)}{\partial \Psi}\frac{\partial l_c(\Psi)^T}{\partial \Psi}\Big|Y\right).
\end{eqnarray}
The elements of this information matrix for the NEF laws are provided in the Appendix.

\section{Simulation}\label{simulation}

In this section we present a small Monte Carlo study for comparing the performance of the EM-algorithm and the method of moments for estimating the parameters of the NEF laws. We also check the estimation of the standard errors obtained from the observed information matrix via EM-algorithm. 

We consider the cases where data are generated from the normal-gamma and NIG distributions. To generate from these distributions, we use the stochastic representation $Y \stackrel{d}{=} \mu W_{\phi} + \sigma \sqrt{W_{\phi}}Z$, where $Z\sim N(0,1)$ independent of $W_{\phi}$, which is $\mbox{Gamma}(\phi)$ or $\mbox{IG}(\phi)$ distributed, respectively. We set the true parameter vector $\Psi = (\mu, \sigma^2, \phi) = (3, 4, 2)$ and sample sizes $n = 30,50,100,150,200,500,1000$. We run a Monte Carlo simulation with $5000$ replicas. Further, we use the MM estimates as initial guesses for the EM-algorithm and consider its convergence criterion to be the one proposed in Subsection \ref{secaoEM} with $\epsilon=10^{-4}$.

Figures \ref{boxplots_parNG} and \ref{boxplots_parNIG} present boxplots of the estimates of the parameters based on the EM-algorithm and method of moments for some sample sizes under the normal gamma and NIG distributions, respectively. Overall, the bias and variance of the estimates go to 0 as the sample size increases, as expected. Let us now  discuss each case with more details.

Concerning the parameter $\mu$, both methods yield similar results under normal gamma and NIG assumptions. On the other hand, regarding the estimation of the parameters $\sigma^2$ and $\phi$, the EM-algorithm has a superior performance over the method of moments in all cases considered for both normal gamma and NIG distributions. We observe that the method of moments yields a considerable bias, even for sample sizes $n=200,500,1000$, in constrast with EM-approach which produces unbiased estimates even for sample sizes $n=50,100$.

Another problem of the MM estimator is that it can produce estimates out of the parameter space.
Under the normal gamma distribution, the percentages of negative estimates for $\sigma^2$ and/or $\phi $ with sample sizes $n = 30,50,100,150,200,500,1000$ were respectively $6.34\%$, $4.72\%$, $3.18\%$, $2.5\%$, $0.24\%$, $0.18\%$ and $0.24\%$. The respective percentages for the NIG case were respectively $3.12\%$, $1.34\%$, $0.74\%$, $0.66\%$, $0.62\%$, $0.24\%$ and $0.04\%$. In these cases, the Monte Carlo replicas were discarted and new values were generated. It is worth to mention that some huge outlier estimates were yielded by the method of moments (for small sample sizes). They cannot be seen from the plots due to the scale of the boxplots, which were chosen to give a clear view of the big picture. 

We finish this section by presenting the estimation of the standard errors of the EM-estimates based on the information matrix given in (\ref{infmatrix}). Tables \ref{est_stand_NG} and \ref{est_stand_NIG} show the standard error of the estimates of the parameters (empirical) and the mean of the standard errors obtained from the information matrix (theoretical) for normal gamma and NIG cases, respectively, for some sample sizes. From these tables, we observe a good agreement between the empirical and the estimated theoretical standard errors, mainly for sample sizes $n\geq100$, for both normal gamma and NIG distributions.

\begin{figure}[H]
\centering
\includegraphics[width=\linewidth, height = 0.30\textheight, keepaspectratio]{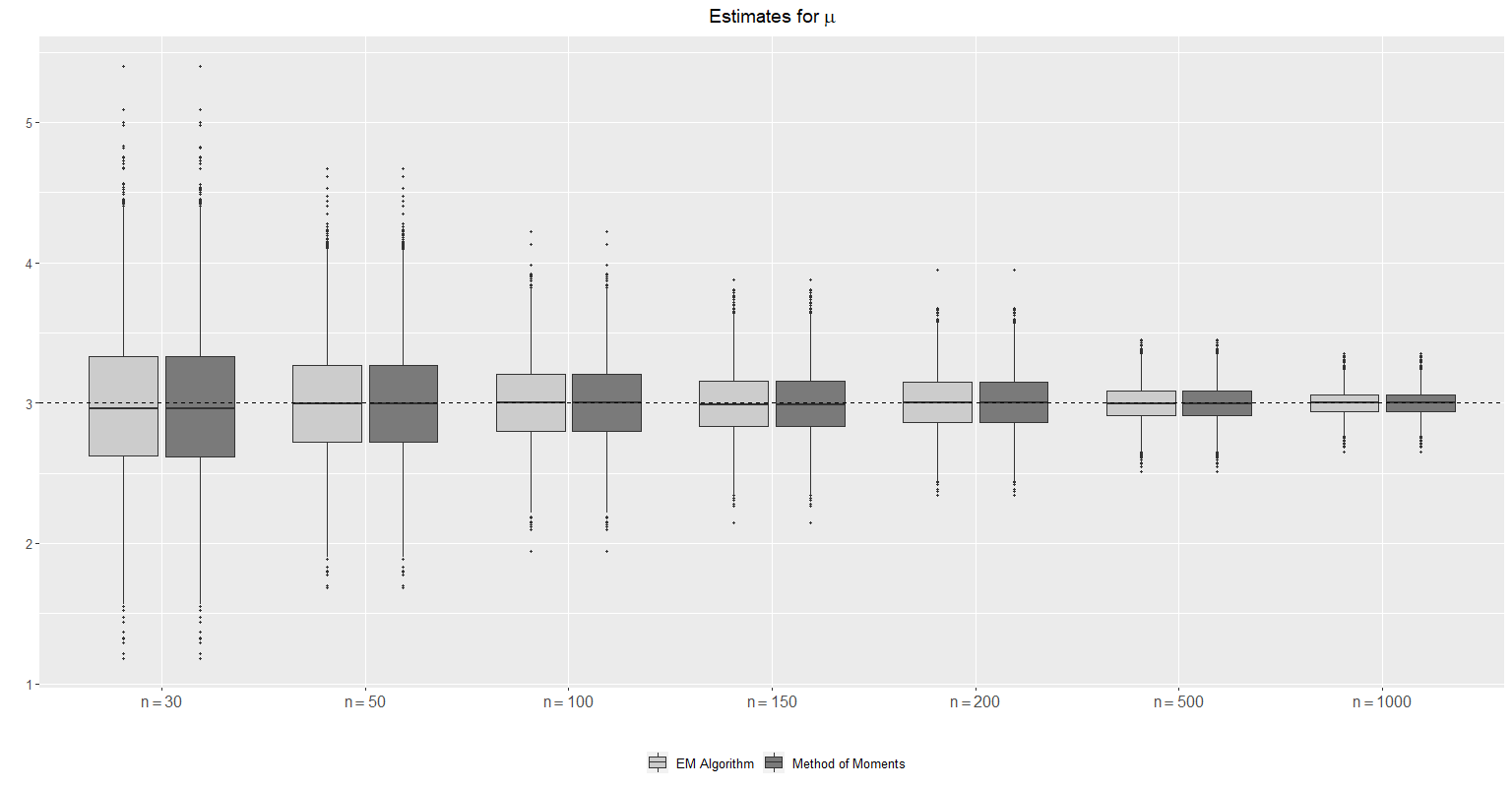}
\includegraphics[width=\linewidth, height = 0.30\textheight, keepaspectratio]{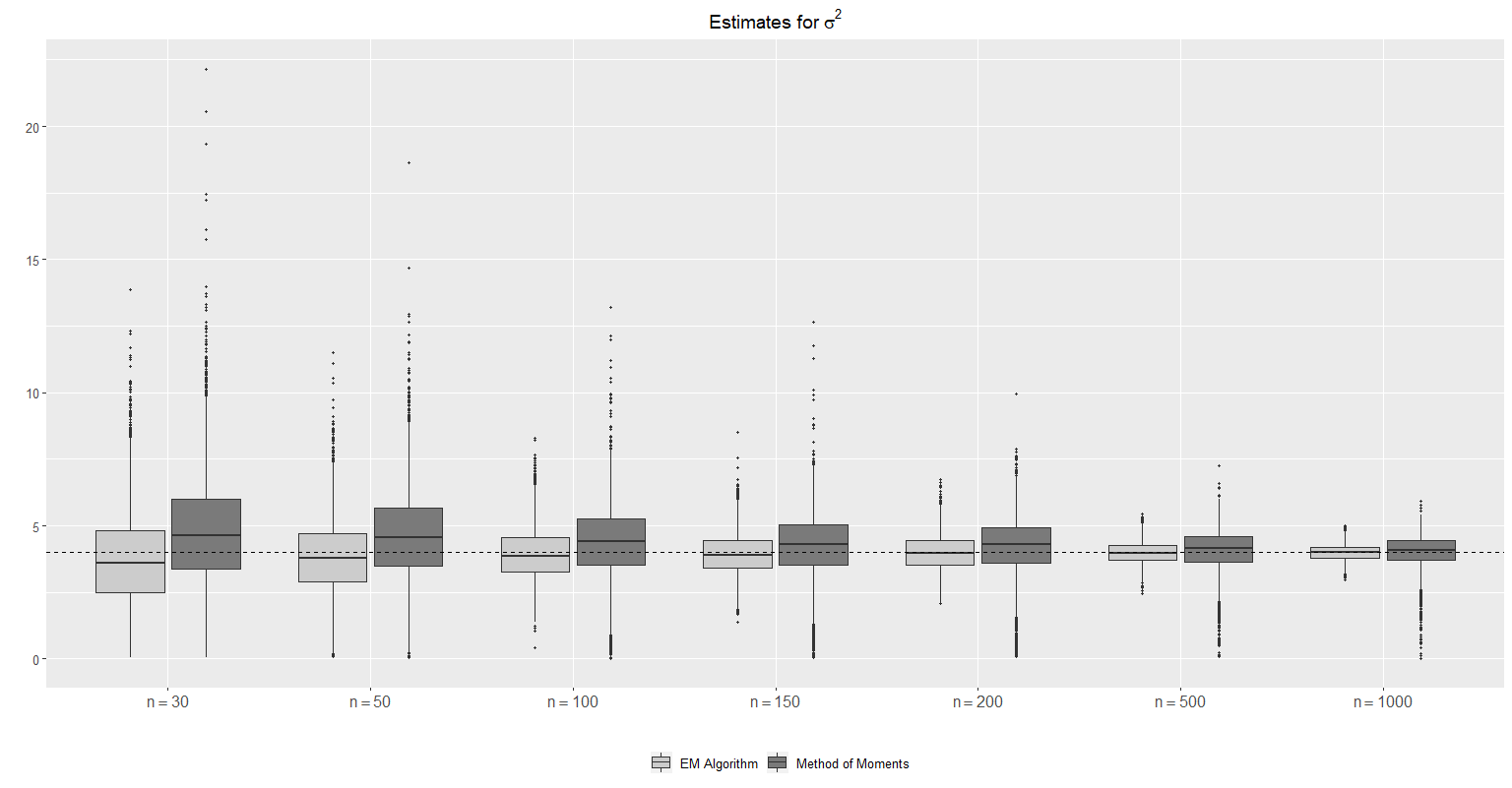}
\includegraphics[width=\linewidth, height = 0.30\textheight, keepaspectratio]{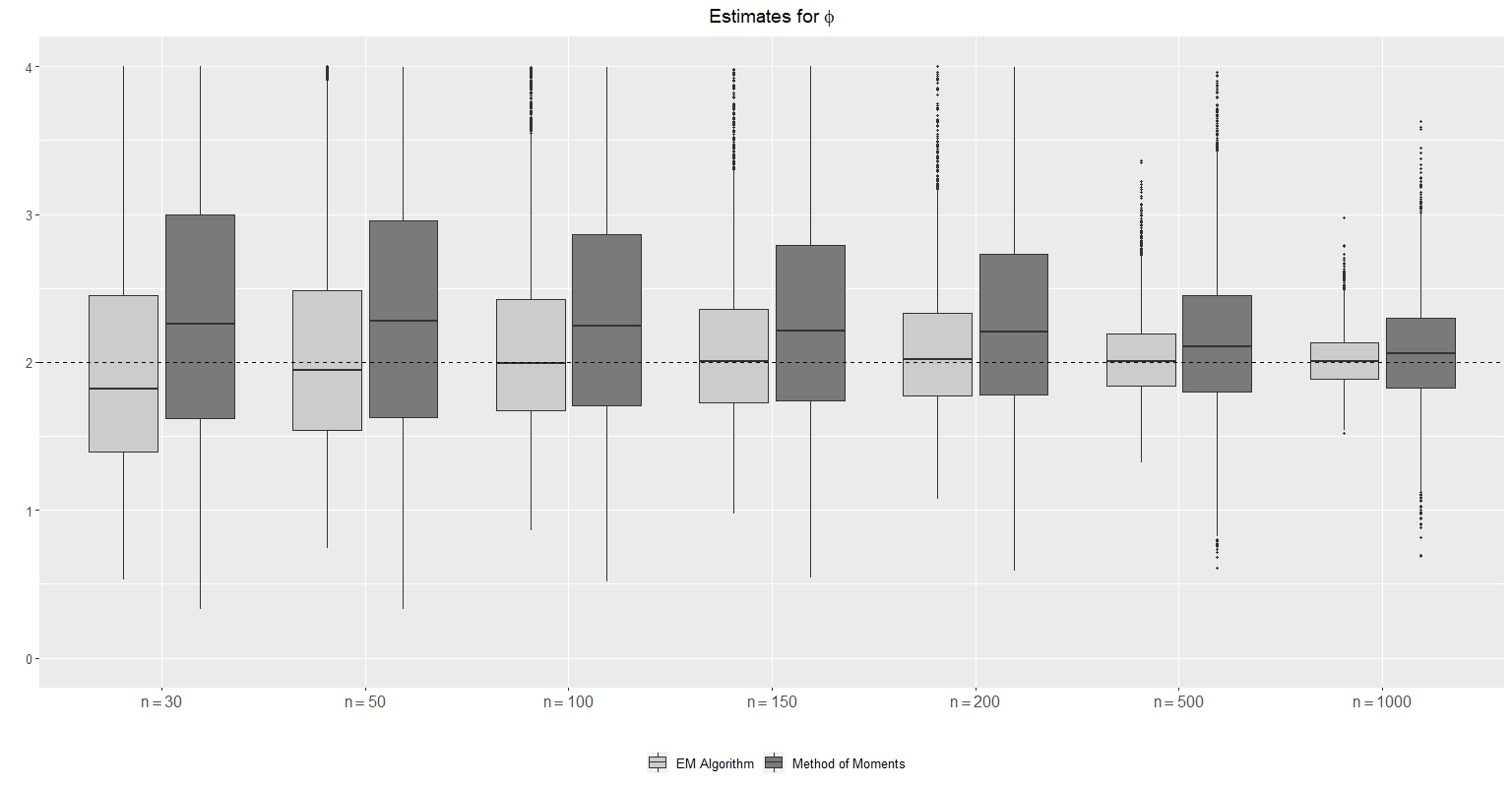}
\caption{Boxplots with the estimates of $\mu$, $\sigma^2$ and $\phi$ obtained based on the EM-algorithm and method of moments under normal gamma distribution. Dotted horizontal lines indicate the true value of the parameter.}
\label{boxplots_parNG}
\end{figure}

\begin{figure}[H]
\centering
\includegraphics[width=\linewidth, height = 0.30\textheight, keepaspectratio]{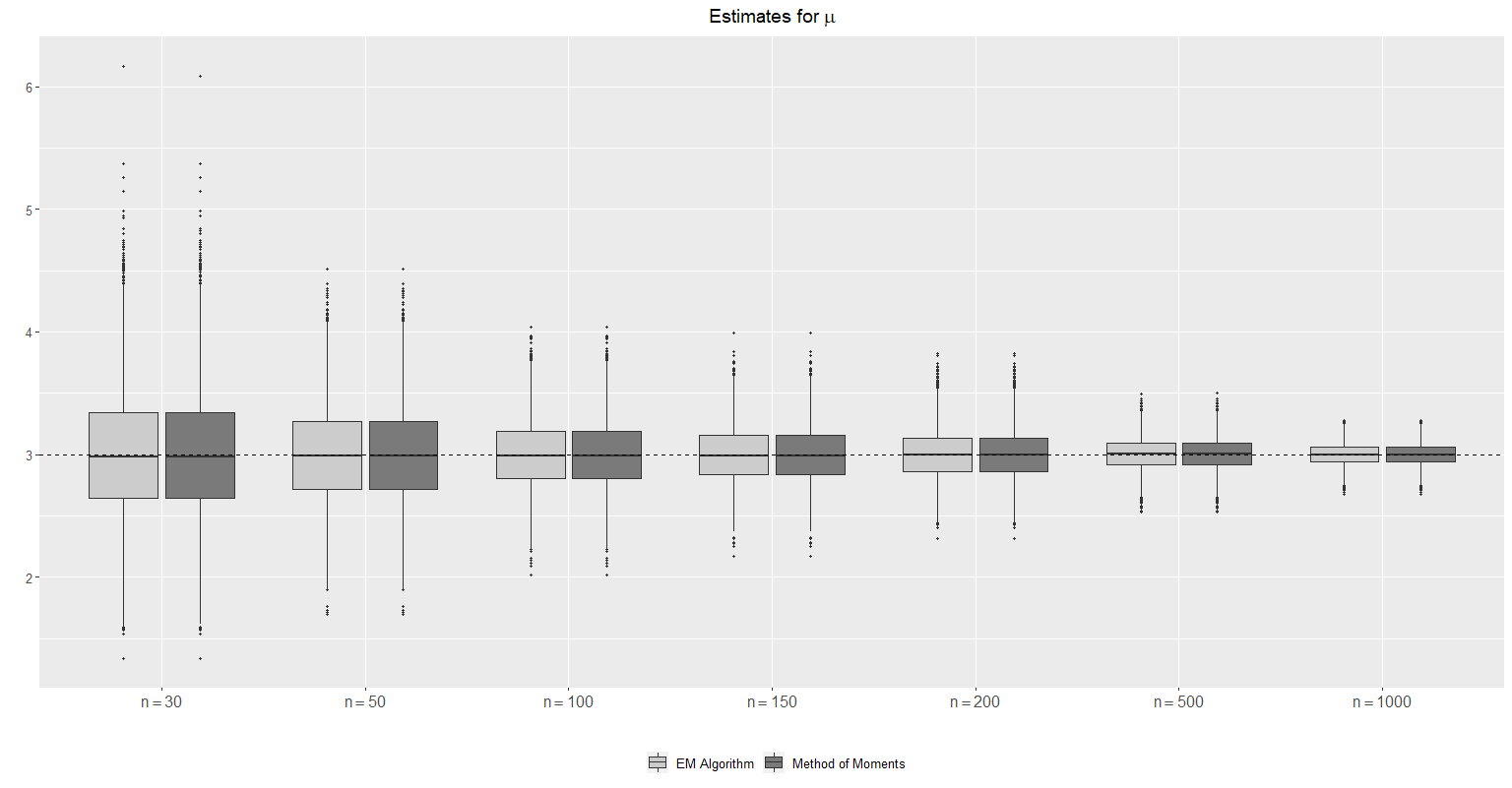}
\includegraphics[width=\linewidth, height = 0.30\textheight, keepaspectratio]{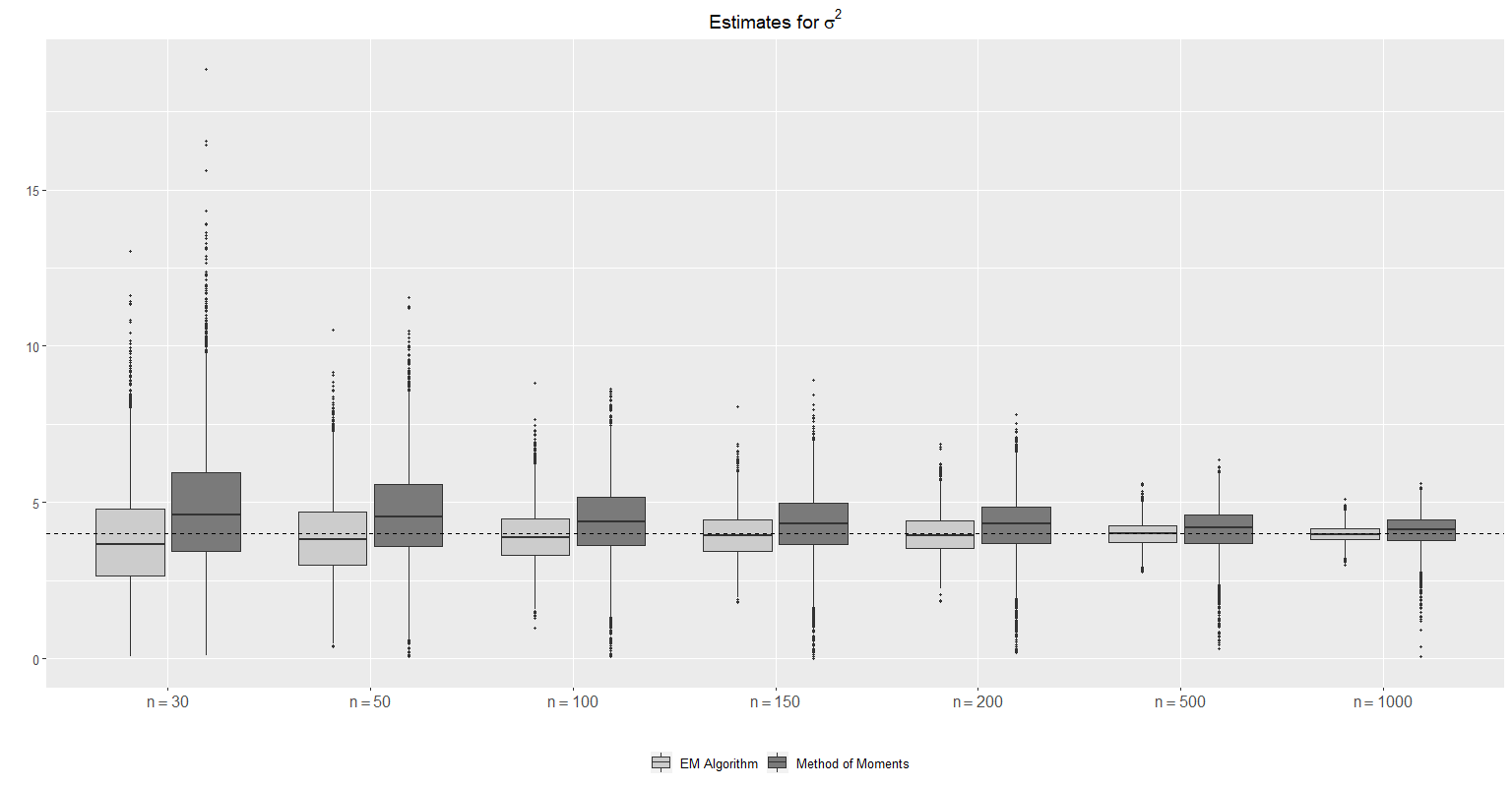}
\includegraphics[width=\linewidth, height = 0.30\textheight, keepaspectratio]{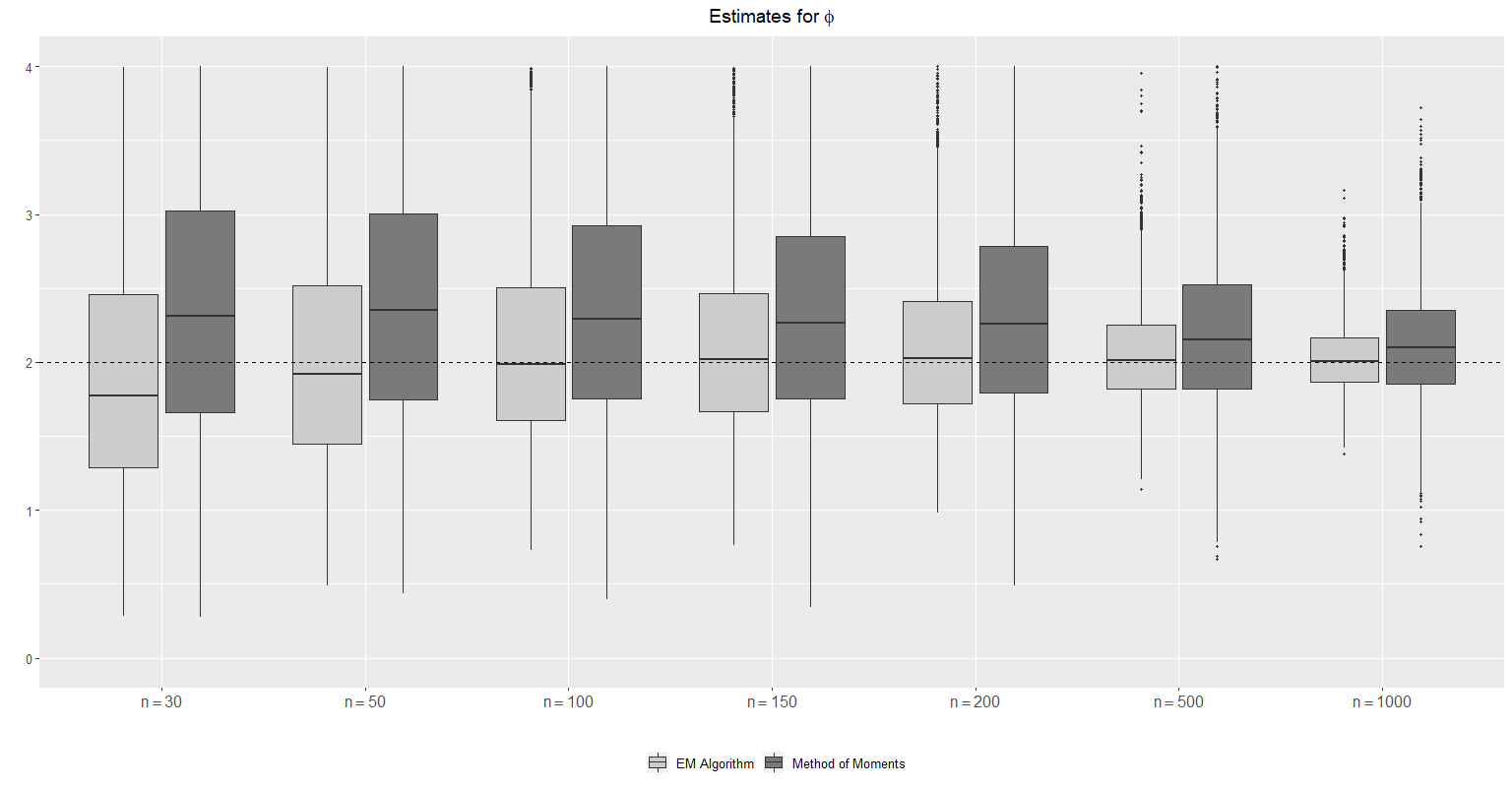}
\caption{Boxplots with the estimates of $\mu$, $\sigma^2$ and $\phi$ obtained based on the EM-algorithm and method of moments under normal inverse-Gaussian distribution. Dotted horizontal lines indicate the true value of the parameter.}
\label{boxplots_parNIG}
\end{figure}

\begin{table}[ht]\footnotesize
\caption{Empirical and theoretical standard errors of parameter estimates under normal gamma assumption.}
\label{est_stand_NG}
\centering
\begin{tabular} {ccccc}
\toprule
& & $\mu$ & $\sigma^2$ & $\phi$ \\\midrule
$n = 30$ & Empirical &   0.5251 & 1.7598 & 6.3974\\
 & Theoretical &  0.5227 & 1.6458 & 4.9593   \\ \midrule
 $n = 50$ & Empirical & 0.4144 & 1.3670 & 3.6329  \\
 & Theoretical & 0.4073 & 1.2927 & 2.1497   \\ \midrule
 $n = 100$& Empirical &   0.2937 & 0.9669 & 0.8815 \\
 & Theoretical & 0.2899 & 0.9242 & 0.7585    \\ \midrule
$n = 150$& Empirical & 0.2376 & 0.7768 & 0.6010 \\
& Theoretical & 0.2368 & 0.7533 & 0.5414 \\\midrule
$n = 200$& Empirical & 0.2096 & 0.6723 & 0.4754 \\
& Theoretical & 0.2056 & 0.6558 & 0.4529 \\\midrule
$n = 500$ & Empirical & 0.1313 & 0.4195 & 0.2727 \\
& Theoretical & 0.1302 & 0.4159 & 0.2658 \\\midrule
$n = 1000$& Empirical & 0.0910 & 0.2957 & 0.1851 \\
& Theoretical & 0.0922 & 0.2948 & 0.1846 \\\bottomrule
\end{tabular}
\end{table}

\begin{table}[H]\footnotesize
\caption{Empirical and theoretical standard errors of parameter estimates under normal inverse-Gaussian assumption.}
\label{est_stand_NIG}
\centering
\begin{tabular}{ccccc}
\toprule
& & $ \mu $ & $ \sigma ^ 2 $ & $ \phi $ \\\midrule
  $n = 30$ & Empirical & 0.5349 & 1.6942 & 21.5352   \\
 & Theoretical &  0.5287 & 1.6094 & 21.2078 \\ \midrule
 $n = 50$& Empirical &  0.4067 & 1.3139 & 11.3969  \\
 & Theoretical & 0.4099 & 1.2545 & 6.1186  \\ \midrule
$n = 100$ & Empirical & 0.2890 & 0.8872 & 1.2192    \\
 & Theoretical & 0.2901 & 0.8849 & 0.9936   \\ \midrule
$n = 150$& Empirical & 0.2376 & 0.7379 & 0.7910 \\
& Theoretical & 0.2374 & 0.7280 & 0.6852 \\\midrule
$n = 200$& Empirical & 0.2083 & 0.6385 & 0.6106 \\
& Theoretical & 0.2057 & 0.6309 & 0.5596 \\\midrule
$n = 500$& Empirical & 0.1325 & 0.4013 & 0.3391 \\
& Theoretical & 0.1303 & 0.4000 & 0.3266 \\\midrule
$n = 1000$& Empirical & 0.0903 & 0.2765 & 0.2295 \\
& Theoretical & 0.0921 & 0.2827 & 0.2254 \\\bottomrule
\end{tabular}
\end{table}

\section{Real data application}\label{application}

We here apply the NEF laws and the proposed EM-algorithm in a real data to illustrate their usefulness in practical situations. We consider daily log-returns of Petrobras stock from Jan 1st 2010 to Dec 31th 2018, which consists of 2263 observations. These data can be obtained through the website \url{https://finance.yahoo.com/}. Denote $P_t$ being the stock price at time $t$ and $Y_t=\log(P_t/P_{t-1})$ the log-return, for $t=1,\ldots,n$, where $n$ denotes the sample size; in this application, $n=2263$. According to \cite{sc}, if $N_\lambda$ is the number of market transactions in an interval of time (one day, for example), with $\lambda$ denoting the mean number of transactions, each of these transactions have an associated return, here denoted by $X_j$, which are a sequence of {\it i.i.d.} random variables with finite variance. Therefore, under these assumptions, we have that $Y_t=\log(P_t/P_{t-1})=\sum_{j=1}^{N_\lambda}X_j$. The number of daily transactions of the Petrobras stock is high due to its liquidity, in other words, the mean number of transactions $\lambda$ is high. This justifies the modeling of these stocks through a NEF class of distributions due to Theorem \ref{theo_conv}.

\begin{table}[H]
\centering
\caption{Estimates of the parameters with their respective standard errors (in parentheses) for the daily log-return of Petrobras stock prices under normal, normal gamma (NG) and NIG models.}
\label{petrobras results}
\begin{tabular}{c|ccc}
\hline
Model & \multicolumn{3}{c}{Estimates} \\ 
 & $ \mu $ & $ \sigma^ 2 $ & $ \phi $ \\
 \hline
Normal & $-$0.0006 (0.0007) & 0.0010 (0.0001) & $\infty$ \\
NG & $-$0.0006 (0.0006) & 0.0009 (0.0001) & 1.3105 (0.1105) \\
NIG & $-$0.0006 (0.0006) & 0.0010 (0.0001) & 0.8201 (0.1161) \\
\hline
\end{tabular}
\end{table}

Table \ref{petrobras results} shows the maximum likelihood estimates of the parameters based on the normal distribution and the EM-estimates for the normal gamma and NIG models. The standard errors are also provided in this table. All models provide similar estimates for the mean and scale parameters as expected. We emphasize that the parameter $\phi$ controls the departure from the normal distribution. By taking $\phi\rightarrow\infty$, we obtain the normal law as a limiting case of the NEF class. The estimates for this parameter under both normal gamma and NIG models indicate some departure from the normal distribution.

This comment is better supported by Figure \ref{petro histogram}, which provides the histogram of the data with the estimated densities of the normal, normal gamma and NIG laws. We can observe that the normal gamma and NIG densities capture well the peak, in contrast with the normal density.

\begin{figure}[H]
\centering
\includegraphics[width=\linewidth, height = 0.35\textheight, keepaspectratio]{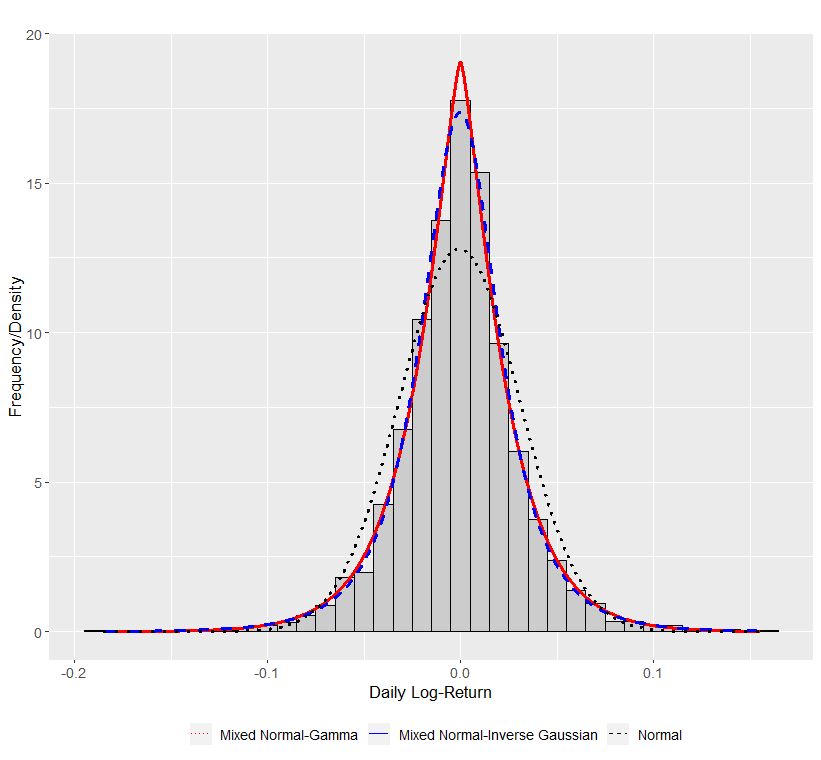}
\caption{Histogram of daily log-returns of Petrobras stock price with fitted normal, NG and NIG densities.}
\label{petro histogram}
\end{figure}

\begin{figure}[H]
\centering
\subfigure[Normal qq-plot]{\includegraphics[width=\textwidth,height=0.2\textheight,keepaspectratio]{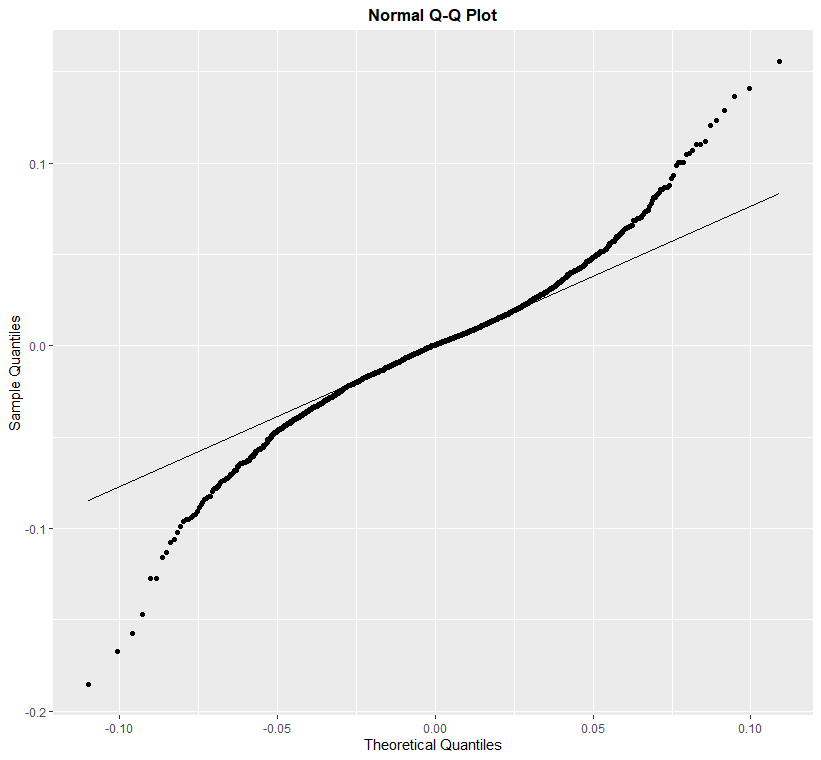}}
\subfigure[NG qq-plot]{\includegraphics[width=\textwidth,height=0.2\textheight,keepaspectratio]{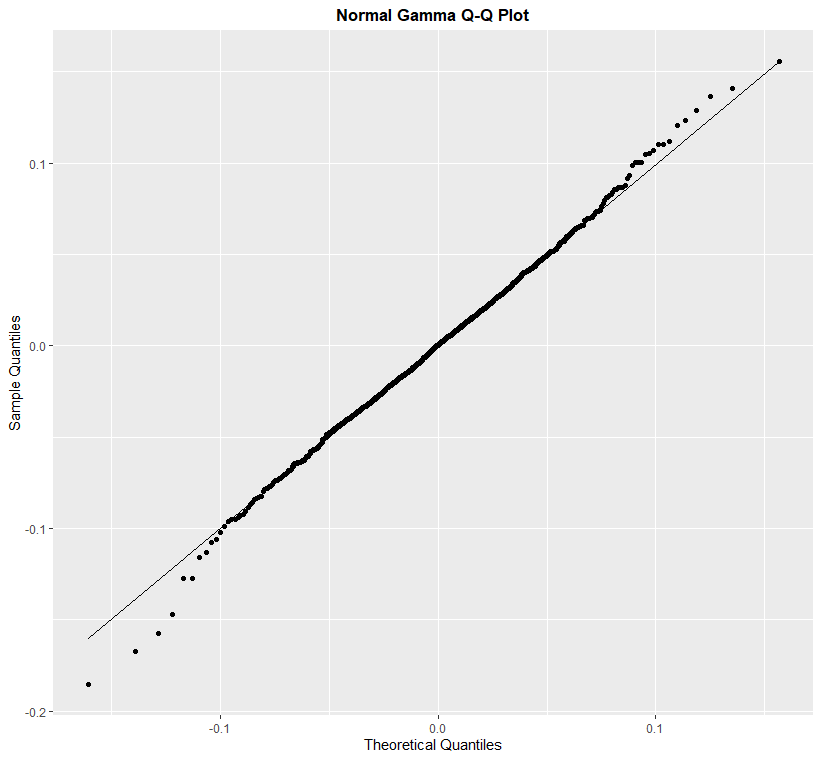}}
\subfigure[NIG qq-plot]{\includegraphics[width=\textwidth,height=0.2\textheight,keepaspectratio]{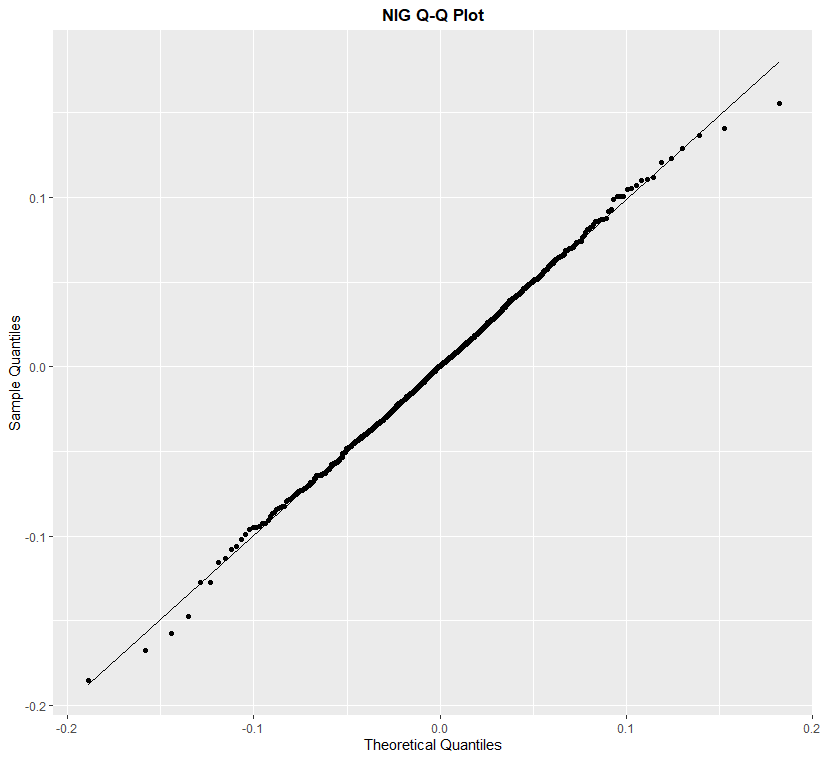}}
\caption{QQ plots of the daily log-returns of Petrobras stock for the fitted normal, normal gamma and NIG laws.}
\label{qqplotPetro}
\end{figure}

To check goodness-of-fit of the models, we consider qq-plots, which consist in plotting the empirical quantiles against the fitted ones. A well-fitted model provides a qq-plot looking like a linear function $y=x$. Figure \ref{qqplotPetro} exhibits the qq-plots based on the normal, NG and NIG fitted models. From this figure, we clearly observe that the normal distribution is not suitable for modeling the log-returns, which was already expected. We also notice that both NG and NIG models provide satisfactory fits, being the last one capturing better the tails. We stress that the modeling of tails is an important task in the study of financial data. For this particular dataset presented in this section, we recommend the use of the NIG law.

\section*{Appendix}

Define the conditional expectations $\alpha_i = E\left({W_{\phi}}_i|Y_i\right)$, $\gamma_i = E\left({{W_{\phi}}_i}^{-1}\big|Y_i\right)$, $\delta_i = E\left(g({W_{\phi}}_i)\big|Y_i\right)$,  $\lambda_i = E\left({W_{\phi}}_i^2\big|Y_i\right)$, $\tau_i = E\left({W_{\phi}}_ig({W_{\phi}}_i)\big|Y_i\right)$, $\nu_i = E\left(g({W_{\phi}}_i)^2\big|Y_i\right)$, $\rho_i = E\left({{W_{\phi}}_i}^{-2}\big|Y_i\right)$ and $\varphi_i = E\left({{W_{\phi}}_i}^{-1}g({W_{\phi}}_i)\big|Y_i\right)$, for $i=1,\ldots,n$.

The elements of the observed information matrix are given by the following expressions:

\begin{eqnarray*}
E\left(-\dfrac{\partial^2l_c}{\partial \mu^2}\Big|Y\right) = \dfrac{1}{\sigma^2} \displaystyle \sum\limits_{i = 1}^{n}\alpha_i, \quad E\left(-\dfrac{\partial^2l_c}{\partial (\sigma^2)^2 }\Big|Y\right) = \dfrac{1}{(\sigma^2)^3} \displaystyle \sum\limits_{i = 1}^{n}\Big\{y_i^2\gamma_i - 2 \mu y_i + \mu^2 \alpha_i\Big\} - \dfrac{n}{2(\sigma^2)^2},
\end{eqnarray*}
\begin{eqnarray*}
E\left(-\dfrac{\partial^2l_c}{\partial \phi^2 }\Big|Y\right) = -n d''(\phi), \quad 
E\left(-\dfrac{\partial^2l_c}{\partial \mu \partial \sigma^2 }\Big|Y\right) =  \dfrac{1}{(\sigma^2)^2} \displaystyle \sum\limits_{i = 1}^{n}\{y_i - \mu\alpha_i\},\quad E\left(-\dfrac{\partial^2l_c}{\partial \mu \partial \phi }\Big|Y\right)=0,
\end{eqnarray*} 
\begin{eqnarray*}
&&E\left(\left(\dfrac{\partial l_c}{\partial \mu}\right)^2 \Big|Y \right) = \dfrac{1}{(\sigma^2)^2}\left[\left(\sum\limits_{i = 1}^{n}y_i\right)^2 - 2\mu\sum\limits_{i = 1}^{n}y_i\sum\limits_{i = 1}^{n}\alpha_i + \mu^2\left(\sum\limits_{i = 1}^{n}\lambda_i + \sum\limits_{i \neq j}\alpha_i\alpha_j\right)\right],
\end{eqnarray*} 

\begin{eqnarray*} 
&&E\left(\left(\dfrac{\partial l_c}{\partial \sigma^2 }\right)^2 \Big|Y \right)=\frac{\mu^4}{4(\sigma^2)^4}\left(\sum\limits_{i = 1}^{n}\lambda_i + \sum\limits_{i \neq j}\alpha_i\alpha_j\right)+\left(-\frac{n}{2\sigma^2} - \frac{\mu}{(\sigma^2)^2}\sum\limits_{i = 1}^{n}y_i\right)\\
&&\times\left(-\frac{n}{2\sigma^2} + \frac{1}{(\sigma^2)^2}\sum\limits_{i = 1}^{n}\Big\{y_i^2\gamma_i - \mu y_i + \mu^2\alpha_i\Big\}\right)+ 
\frac{1}{4(\sigma^2)^4}\left(\sum\limits_{i = 1}^{n}y_i^4\rho_i + \sum\limits_{i \neq j}y_i^2\gamma_i\; y_j^2 \gamma_j\right)\\
&&+\frac{\mu^2}{2(\sigma^2)^4}\left(\sum\limits_{i = 1}^{n}y_i^2 + \sum\limits_{i \neq j}\alpha_iy_j^2\gamma_j\right),\quad E\left(-\dfrac{\partial^2l_c}{\partial \sigma^2 \partial \phi }\Big|Y\right)=0,
\end{eqnarray*} 

\begin{eqnarray*}
&&E\left(\left(\dfrac{\partial l_c}{\partial \phi}\right)^2 \Big|Y \right)=n\left(-b(\xi_0) + d'(\phi)\right)\left(-nb({\xi}_0) + nd'(\phi) + 2\sum\limits_{i = 1}^n\Big\{\delta_i + \xi_0\alpha_i\Big\}\right)+\sum\limits_{i = 1}^n\nu_i +\\
&&\sum\limits_{i \neq j}\delta_i\delta_j +2\xi_0\left(\sum\limits_{i = 1}^n\tau_i + \sum\limits_{i \neq j}\alpha_i\delta_j\right) + \xi_0^2\left(\sum\limits_{i = 1}^n\lambda_i + \sum\limits_{i \neq j}\alpha_i\alpha_j\right),
\end{eqnarray*} 
 
\begin{eqnarray*}    
&&E\left(\dfrac{\partial l_c}{\partial \mu} \dfrac{\partial l_c}{\partial \sigma^2} \Big|Y \right) = \dfrac{1}{\sigma^2}\sum\limits_{i = 1}^{n}y_i\left(-\frac{n}{2\sigma^2} + \frac{1}{2(\sigma^2)^2}\sum\limits_{j = 1}^{n}\left\{y_j^2\gamma_j - 2\mu y_j + \mu^2\alpha_j\right\}\right)\\
&&+ \frac{\mu}{\sigma^2}\sum\limits_{i = 1}^{n}\alpha_i\left(\frac{n}{2\sigma^2}+ \frac{\mu}{(\sigma^2)^2}\sum\limits_{j = 1}^{n}y_j\right)-\frac{\mu}{2(\sigma^2)^3}\left(\sum\limits_{i = 1}^{n}y_i^2 + \sum\limits_{i \neq j}\alpha_iy_j^2\gamma_j\right) - \frac{\mu^3}{2(\sigma^2)^3}\bigg(\sum\limits_{i = 1}^{n}\lambda_i +\\
&&\sum\limits_{i \neq j}\alpha_i\alpha_j\bigg),\quad E\left(\dfrac{\partial l_c}{\partial \mu} \dfrac{\partial l_c}{\partial \phi} \Big|Y \right)=\dfrac{n}{\sigma^2}\sum\limits_{i = 1}^{n}\left\{y_i - \mu\alpha_i\right\}(-b(\xi_0) + d'(\phi)) +\\ &&\dfrac{1}{\sigma^2}\sum\limits_{i = 1}^{n}y_i\sum\limits_{j = 1}^{n}\Big\{\delta_j + \xi_0\alpha_j\Big\}-\frac{\mu}{\sigma^2}\left(\sum\limits_{i = 1}^{n}\tau_i + \sum\limits_{i \neq j}\alpha_i\delta_j\right) - \frac{\mu\xi_0}{\sigma^2}\left(\sum\limits_{i = 1}^{n}\lambda_i + \sum\limits_{i \neq j}\alpha_i\alpha_j\right)
\end{eqnarray*}     
and
\begin{eqnarray*}
&& E\left(\dfrac{\partial l_c}{\partial \sigma^2} \dfrac{\partial l_c}{\partial \phi} \Big|Y \right) =\left(-\frac{n}{2\sigma^2} - \frac{\mu}{(\sigma^2)^2}\sum\limits_{i = 1}^n y_i\right)\left(-n b(\xi_0) + n d'(\phi) + \sum\limits_{i = 1}^n\Big\{\delta_i + \xi_0\alpha_i\Big\}\right) +\\
&& \frac{n}{2(\sigma^2)^2}\sum\limits_{i = 1}^n\Big\{y_i^2\gamma_i + \mu^2\alpha_i\Big\}\left(- b(\xi_0) +  d'(\phi)\right) + \frac{1}{2(\sigma^2)^2}\left(\sum\limits_{i = 1}^n y_i^2\varphi_i + \sum\limits_{i \neq j}y_i^2\gamma_i\delta_j\right) +\\
&& \frac{\xi_0}{2(\sigma^2)^2}\left(\sum\limits_{i = 1}^ny_i^2 + \sum\limits_{i \neq j}y_i^2\gamma_i\alpha_j\right) + \frac{\mu^2}{2(\sigma^2)^2}\left(\sum\limits_{i = 1}^n \tau_i+ \sum\limits_{i \neq j}\alpha_i\delta_j\right) + \frac{\mu^2\xi_0}{2(\sigma^2)^2}\left(\sum\limits_{i = 1}^n \lambda_i+ \sum\limits_{i \neq j}\alpha_i\alpha_j\right).
\end{eqnarray*}

We now provide explicit expressions for the conditional expectations involved in the information matrix for the normal gamma and NIG cases. We omit the index $j$ to simplify the notation.

\begin{example} Let $Y$ be a normal gamma distribution with parameters $\mu \in \mathbb{R}$, $\sigma^2 > 0$, $\phi > 0$ with associated latent factor $W_{\phi} \sim \textrm{Gamma}(\phi)$. Define $a = \frac{\mu^2}{\sigma^2} + 2\phi$ and $b = \frac{y^2}{\sigma^2}$. Then, we have that
\begin{eqnarray*}
\lambda = E\left(W_{\phi}^2\big|Y\right) = \dfrac{\mathcal{K}_{\phi + \frac{3}{2}}(\sqrt{ab})}{\mathcal{K}_{\phi - \frac{1}{2}}(\sqrt{ab})} \dfrac{b}{a},\quad \rho = E\left({W_{\phi}}^{-2}\big|Y\right) = \dfrac{\mathcal{K}_{\phi - \frac{5}{2}}(\sqrt{ab})}{\mathcal{K}_{\phi - \frac{1}{2}}(\sqrt{ab})}\dfrac{a}{b},
\end{eqnarray*}
\begin{eqnarray*}
\tau = E\left(W_{\phi}\log W_{\phi}\big|Y\right) = \dfrac{\mathcal{K}_{\phi + \frac{1}{2}}(\sqrt{ab})}{\mathcal{K}_{\phi - \frac{1}{2}}(\sqrt{ab})} \left(\dfrac{b}{a}\right)^{\frac{1}{2}}\left(\dfrac{1}{2}\log\left(\dfrac{b}{a}\right) + \dfrac{\partial}{\partial \phi}\log\mathcal{K}_{\phi + \frac{1}{2}}(\sqrt{ab})\right),
\end{eqnarray*}
\begin{eqnarray*}
\nu = E\left((\log W_{\phi})^2\big|Y\right) = E_U\left((\log U)^2\right), \; \mathrm{with} \; U \sim GIG\left(a, b, \phi - \frac{1}{2}\right)   
\end{eqnarray*}
and
\begin{eqnarray*}
\varphi = E\left({W_{\phi}}^{-1}\log W_{\phi}\big|Y\right) = \dfrac{\mathcal{K}_{\phi - \frac{3}{2}}(\sqrt{ab})}{\mathcal{K}_{\phi - \frac{1}{2}}(\sqrt{ab})}\left(\dfrac{a}{b}\right)^{\frac{1}{2}}\left(\dfrac{1}{2}\log\left(\dfrac{b}{a}\right) + \dfrac{\partial}{\partial \phi}\log\mathcal{K}_{\phi - \frac{3}{2}}(\sqrt{ab})\right).
\end{eqnarray*}
\end{example}

\begin{example} For the NIG case, the conditional expectations assume the forms
\begin{eqnarray*}
\lambda = E\left(W_{\phi}^2\Big|Y\right) = \dfrac{b}{a}, \quad \rho = E\left({W_{\phi}}^{-2}\Big|Y\right) = \dfrac{a}{b}\dfrac{\mathcal{K}_{-3}(\sqrt{ab})}{\mathcal{K}_{-1}(\sqrt{ab})},  
\end{eqnarray*}
\begin{eqnarray*}
\tau = E\left(W_{\phi}\left(-\dfrac{1}{2W_{\phi}}\right)\Big|Y\right) = - \dfrac{1}{2},\quad
\nu = E\left(\left(-\dfrac{1}{2W_{\phi}}\right)^2\Big|Y\right) =  \dfrac{a}{4b}\dfrac{\mathcal{K}_{-3}(\sqrt{ab})}{\mathcal{K}_{-1}(\sqrt{ab})}
\end{eqnarray*}     
and 
\begin{eqnarray*}
\varphi = E\left({W_{\phi}}^{-1}\left(-\dfrac{1}{2W_{\phi}}\right)\Big|Y\right) = \dfrac{a}{2b}\dfrac{\mathcal{K}_{-3}(\sqrt{ab})}{\mathcal{K}_{-1}(\sqrt{ab})}.
\end{eqnarray*}     
\end{example}

\end{document}